\documentclass[11pt]{article}
\usepackage{latexsym,amsfonts,amssymb,amsmath,amsthm}
\usepackage{graphicx}

\usepackage[usenames,dvipsnames]{color}
\usepackage{ulem,comment}

\usepackage{color}

\title{Persistence and spreading speeds of parabolic-elliptic Keller-Segel models in shifting environments}

\author{
 Wenxian Shen\thanks{Partially supported by the NSF grant DMS--1645673}\,\, and\,\,  Shuwen Xue\\
Department of Mathematics and Statistics\\
Auburn University, AL 36849 }

\date{}

\begin{document}

\maketitle

\newtheorem{tm}{Theorem}[section]
\newtheorem{prop}{Proposition}[section]
\newtheorem{defin}{Definition}[section] 
\newtheorem{coro}{Corollary}[section]
\newtheorem{lem}{Lemma}[section]
\newtheorem{assumption}{Assumption}[section]
\newtheorem{rk}{Remark}[section]
\newtheorem{nota}[tm]{Notation}
\numberwithin{equation}{section}

\newcommand{\stk}[2]{\stackrel{#1}{#2}}
\newcommand{\dwn}[1]{{\scriptstyle #1}\downarrow}
\newcommand{\upa}[1]{{\scriptstyle #1}\uparrow}
\newcommand{\nea}[1]{{\scriptstyle #1}\nearrow}
\newcommand{\sea}[1]{\searrow {\scriptstyle #1}}
\newcommand{\csti}[3]{(#1+1) (#2)^{1/ (#1+1)} (#1)^{- #1
 / (#1+1)} (#3)^{ #1 / (#1 +1)}}
\newcommand{\RR}[1]{\mathbb{#1}}

\newcommand{\rd}{{\mathbb R^d}}
\newcommand{\ep}{\varepsilon}
\newcommand{\rr}{{\mathbb R}}
\newcommand{\alert}[1]{\fbox{#1}}
\newcommand{\eqd}{\sim}
\def\p{\partial}
\def\R{{\mathbb R}}
\def\N{{\mathbb N}}
\def\Q{{\mathbb Q}}
\def\C{{\mathbb C}}
\def\l{{\langle}}
\def\r{\rangle}
\def\t{\tau}
\def\k{\kappa}
\def\a{\alpha}
\def\la{\lambda}
\def\De{\Delta}
\def\de{\delta}
\def\ga{\gamma}
\def\Ga{\Gamma}
\def\ep{\varepsilon}
\def\eps{\varepsilon}
\def\si{\sigma}
\def\Re {{\rm Re}\,}
\def\Im {{\rm Im}\,}
\def\E{{\mathbb E}}
\def\P{{\mathbb P}}
\def\Z{{\mathbb Z}}
\def\D{{\mathbb D}}
\newcommand{\ceil}[1]{\lceil{#1}\rceil}

\begin{abstract}
The current paper is concerned with the persistence and spreading speeds of the
 following Keller-Segel chemoattraction system in shifting environments,
\begin{equation}\label{abstract-eq1}
\begin{cases}
u_t=u_{xx}-\chi(uv_x)_x +u(r(x-ct)-bu),\quad x\in\R\cr
0=v_{xx}- \nu v+\mu u,\quad x\in\R,
\end{cases}
\end{equation}
where  $\chi$, $b$, $\nu$, and $\mu$ are positive constants, { $c\in\R$ },  $r(x)$ is H\"older continuous, bounded,
$r^*=\sup_{x\in\R}r(x)>0$, $r(\pm \infty):=\lim_{x\to \pm\infty}r(x)$ exist, and $r(x)$ satisfies either $r(-\infty)<0<r(\infty)$, or $r(\pm\infty)<0$.
Assume $b>\chi\mu$ and  $b\ge \big(1+\frac{1}{2}\frac{(\sqrt{r^*}-\sqrt{\nu})_+}{(\sqrt{r^*}+\sqrt{\nu})}\big)\chi\mu$.
In the case that $r(-\infty)<0<r(\infty)$, it is shown that if the moving speed $c>c^*:=2\sqrt{r^*}$,
then the species becomes extinct in the habitat. If the moving speed $ -c^*\leq c<c^*$, then the species will persist and spread along the shifting habitat  at the asymptotic spreading speed $c^*$. If the moving speed $c<-c^*$, then the species will spread in the both directions at the asymptotic spreading speed $c^*$. In the case that $r(\pm\infty)<0$, it is shown that  if $|c|>c^*$, then the species will become extinct in the habitat. If $\lambda_{\infty}$, defined to be the generalized principle eigenvalue of the operator $u\to u_{xx}+cu_{x}+r(x)u$, is negative and the degradation rate $\nu$ of the chemo-attractant is grater than or equal to some number $\nu^*$, then the species will also become extinct in the habitat. If $\lambda_{\infty}>0$, then the species will persist surrounding the good habitat.
\end{abstract}



\section{Introduction}

This work is concerned with the persistence and spreading speeds  of the attraction Keller-Segel chemotaxis models in  shifting environments of the form
\begin{equation}\label{Keller-Segel-eq0}
\begin{cases}
u_t=\Delta u-\nabla \cdot (\chi u \nabla v)+u(r(x-ct)-bu),\quad x\in\R\cr
0 =\Delta v-  \nu v +\mu u,\quad x\in\R,
\end{cases}
\end{equation}
where $ b, \nu,\mu$ and $\chi$ are positive constants, { $c\in\R$} and $u(t,x)$ and $v(t,x)$ represent the densities of a
 mobile species and a chemo-attractant, respectively.  Biologically, the positive constant $\chi$ measures the sensitivity effect on the mobile species by  the chemical substance which is produced overtime by the mobile species;   the reaction term $u(r(x-ct)-bu)$ in the first  equation of \eqref{Keller-Segel-eq0} describes the local dynamics of the mobile species which depends on the density $u$ and on the shifting habitat with a fixed speed $c$; $\nu$  represents the degradation rate of the  chemo-attractant;  and $\mu$ is the rate at which the mobile species produces the chemo-attractant.

 One of the first mathematical models of chemotaxis was proposed by Keller and Segel in their works \cite{KeSe1,KeSe2}.
  Chemotaxis describes the oriented movements of biological cells and organisms in response to chemical gradient which they may produce themselves  over time
   and is crucial for many aspects of behaviour such
as the location of food sources, avoidance of predators and attracting mates, slime mold aggregation,
tumor angiogenesis, and primitive streak formation. Chemotaxis is also crucial in
macroscopic process such as population dynamics and gravitational collapse.
 A lot of literature is concerned with  mathematical analysis of various chemotaxis models, including system \eqref{Keller-Segel-eq0} with $r(\cdot)$ being a constant function.
 The reader is referred to  \cite{KJPainter} for detailed applications of chemotaxis models in a wide range of biological phenomena.
  The reader is also referred to \cite{HiPa, Hor}  for some detailed introduction into the mathematics of Keller-Segel chemotaxis models.

In this paper, we consider \eqref{Keller-Segel-eq0} with $c\not =0$ and $r(\cdot)$ being a sign changing function.
 In particular, we will consider the following two cases:

\medskip

\noindent {\bf Case 1.} {\it Favorable and unfavorable habitats are separated in the sense that  $r(x)$ is H\"older continuous,   the limits $r(\pm\infty):=\lim_{x\to \pm\infty}r(x)$ exist and are finite, and $r(-\infty)<0<r(\infty)$},  $r(-\infty)\leq r(x)\leq r(\infty),\, \forall\, x\in\R$.

\medskip

\noindent {\bf Case 2.} {\it Favorable habitat is surrounded by unfavorable habitat in the sense that  $r(x)$ is H\"older continuous, $\sup_{x\in \R}r(x)>0$,  the limits $r(\pm\infty):=\lim_{x\to \pm\infty}r(x)$ exist and are finite, and
 $r(\pm\infty)<0$, $\min\{r(\infty),r(-\infty)\}\leq r(x),\, \forall\, x\in\R$.

 }

\medskip

   In {\bf Case 1}, $r(x-ct)$ divides the spatial domain into two regions: the region with good-quality habitat suitable for growth \{$x\in\R$: $r(x-ct)>0$\} and the region with poor-quality habitat unsuitable for growth \{$x\in\R$: $r(x-ct)<0$\}. The edge of the habitat suitable for species growth is shifting at a speed $c$.
In {\bf Case 2}, $r(x-ct)$ still divides the spatial domain into two regions: one favorable for growth \{$x\in\R$: $r(x-ct)>0$\} and one unfavorable for growth \{$x\in\R$: $r(x-ct)<0$\}. The favorable habitat is bounded and surrounded by the unfavorable habitat. The favorable habitat is shifting at a speed $c$.
   These kinds of problems come from considering the threats associated with certain environmental change such as global climate change, in particular, the global warming,  and the worsening of the environment resulting from industrialization which lead to the shifting or translating of the habitat ranges \cite{Ph}. {However, some species may benefit from the climate change, that is, their living environment is improved by the climate change \cite{FeHe}. Climate shifting can effect the density of some animals, such as snakes, moths. These animals attract potential mates by chemotaxis.
  The reader is referred to  \cite{WPCM} for a review of ecological responses to recent climate change.

\smallskip

Consider \eqref{Keller-Segel-eq0}. Central problems include whether the species persists in certain regions with fixed bounded size;  whether it spreads into larger and larger  regions, and if so, how fast it spreads.

These central problems have been well investigated  for \eqref{Keller-Segel-eq0} in the case without chemotaxis and with space-time homogeneous growth rate $r>0$, that is,
\begin{equation}\label{fisher-kpp0}
u_t=\Delta u+u(r-bu), \quad x\in\R.
\end{equation}
 Due to the pioneering works of Fisher
\cite{Fisher} and Kolmogorov, Petrowsky, Piskunov \cite{KPP} on traveling wave solutions and take-over properties of \eqref{fisher-kpp0},
\eqref{fisher-kpp0} is also referred to as  the Fisher-KPP equation.
 The following results are well-known about the spreading speed of \eqref{fisher-kpp0}.
For any
nonnegative solution $u(t,x)$ of (\ref{fisher-kpp0}), if at
time $t=0$, $u(0,x)=u_0(x)$ is $\frac{r}{b}$ for $x$ near $-\infty$ and $0$ for $x$ near $ \infty$, then
$$\limsup_{x \ge ct, t\to \infty}u(t,x)=0 \quad \forall \, c>2\sqrt r
$$
and
$$\limsup_{x \le ct, t\to \infty}|u(t,x)-\frac{r}{b}|=0\quad \forall\,  c<2\sqrt r.
$$
In
literature, $c^*_0=2\sqrt r$ is   called the {\it
spreading speed} for \eqref{fisher-kpp0}.  Since the pioneering works by  Fisher \cite{Fisher} and Kolmogorov, Petrowsky,
Piscunov \cite{KPP},  a huge amount of research has been carried out toward the front propagation dynamics of
  reaction diffusion equations of the form,
\begin{equation}
\label{general-fisher-eq}
u_t=\Delta u+u f(t,x,u),\quad x\in\R^N,
\end{equation}
where $f(t,x,u)<0$ for $u\gg 1$,  $\partial_u f(t,x,u)<0$ for $u\ge 0$ (see \cite{ArWe2, BHN,  BeHaNa1, BeHaNa2, Henri1, Fre, FrGa, LiZh, LiZh1, Nad, NoRuXi, NoXi1, She1, She2, Wei1, Wei2, Zla}, etc.).

There are also many works on the persistence and spreading speeds of the Fisher-KPP equation in shifting environments, that is,
\begin{equation}\label{fisher-kpp1}
u_t=\Delta u+u(r(x-ct)-bu), \quad x\in\R,
\end{equation}
as well as various variants of \eqref{fisher-kpp1}.
For example,
in {\bf  Case 1},
Li et al. \cite{LBSF} studied the spatial dynamics of system \eqref{fisher-kpp1} { for the case $b=1$} and they showed that the persistence and spreading dynamics depend on the speed $c$ of the shifting habitat edge and the number $c^*$,  where $c^*=2\sqrt{r(\infty)}$ for \eqref{fisher-kpp1}. More precisely, they proved that if $c>c^*$, then the species will become extinct in the habitat, and if $0<c<c^*$, then the species will persist and spread along the shifting habitat at the asymptotic spreading speed $c^*$. Recently, Hu and Zou \cite{HHXZ} demonstrated that { in the case $b=1$, }for any given speed $c>0$, \eqref{fisher-kpp1} admits a nondecreasing traveling wave solution connecting $0$ and $r(\infty)$ with the speed agreeing to the habitat shifting speed $c$, which accounts for an extinction wave. Regarding the spatial dynamics of nonlocal dispersal equations and the lattice differential equations, we refer the readers to \cite{CHBL} and \cite{LiWaZh}.

In {\bf Case 2},  {Berestycki et al. \cite{BDNZ} proposed to use the following  reaction-diffusion equation with a forced speed $c>0$ to study the influence of climate change on the population dynamics of biological species:
\begin{equation}\label{B-fisher-kpp1}
u_t= Du_{xx}+f(u, x-ct), \quad x\in\R.
\end{equation}
A typical $f$ considered in \cite{BDNZ} is
$$
f(x,u)=\begin{cases}
au(1-\frac{u}{K}),\quad &0\leq x \leq L,\\
-ru,\quad &x<0 \,\ {\rm and }\,\ x>L
\end{cases}
$$
for some positives constants $a, r, K, L$. They first considered this special case and derived an explicit condition for persistence of species by gluing phase portraits. Then they generalized their analysis for more general type of $f$ to derive criteria for the persistence of a species in any region with a moving and spatially varying habitat.
More precisely, they showed that if $\lambda_{\infty}$, defined to be the generalized principle eigenvalue of the operator $u\to u_{xx}+cu_{x}+f_{u}(x,0)u$, is less than or equal to zero, then the species will go extinct in the long run. If $\lambda_{\infty}>0$, then the species will persist by traveling along with the shifting climate. The high dimension versions with more general type of $f$ was studied later in
\cite{BhRl1}, \cite{BhRl2}. Potapov and M.A. Lewis \cite{APMAL} studied a similar model \eqref{B-fisher-kpp1} in the context of competing species, where the authors investigated the effect of a moving climate on the outcome of competitive interaction between two species. Zhou and Kot \cite{ZhKo} extended the work in Berestycki et al. \cite{BDNZ} to an integro-difference system. They showed that, for a given growth function, dispersal kernel, and patch size, there is a critical rate of range shift beyond which a species’ population will be driven extinct. Regarding the persistence criterion and inside dynamics of integro-difference equations under the climate change, we refer the readers to the paper \cite{LeMaSh}.  }

Comparing with the Fisher-KPP equation, persistence and  spreading speeds of chemotaxis models have only been studied recently.
For example, Salako and Shen studied the spatial spreading dynamics of \eqref{Keller-Segel-eq0} with constant growth rate $r>0$
 and obtained several fundamental results. Some lower and upper bounds for the propagation speeds of solutions with compactly supported initial functions were derived, and some lower bound for the speeds of traveling wave solutions was also derived. It is proved that all these bounds converge to the spreading speed $c_0^*=2\sqrt r$ of \eqref{fisher-kpp0} as $\chi\to 0$  (see \cite{SaSh3}, \cite{SaSh2}, \cite{SaSh1}).
 The reader is also referred to  \cite{FhCh} for the lower and upper bounds of propagation speeds of \eqref{Keller-Segel-eq0} with constant growth rate.
Very recently, the authors of the  paper \cite{SaShXu} improved the results in \cite{SaSh2}.
It is proved that in the case of constant growth rate $r>0$, if $b>\chi\mu$ and $\big(1+\frac{1}{2}\frac{(\sqrt{r}-\sqrt{\nu})_+}{(\sqrt{r}+\sqrt{\nu})}\big)\chi\mu { \leq} b$ hold, then $2\sqrt r$ is the spreading speed of the solutions of   \eqref{Keller-Segel-eq0} with nonnegative continuous initial function $u_0$ with nonempty compact support, that is,
 $$
 \lim_{t\to\infty}\sup_{|x|\ge ct}u(t,x;u_0)=0\quad \forall\, c>2\sqrt r
 $$
 and
 $$
 \liminf_{t\to\infty}\inf_{|x|\le ct} u(t,x;u_0)>0\quad \forall \, 0<c<2\sqrt r,
 $$
 where $(u(t,x;u_0),v(t,x;u_0))$ is the unique global classical solution of   \eqref{Keller-Segel-eq0} with $u(0,x;u_0)=u_0(x)$.

 Up to the authors' knowledge, there is no study on the persistence and spreading speeds of \eqref{Keller-Segel-eq0} with
 $c\not =0$ and $r(\cdot)$ being a sign changing function.
  The objective of the current paper is to investigate the spatial spreading dynamics of \eqref{Keller-Segel-eq0} with the presence of the chemotaxis and shifting environments. In particular,  we investigate the effects of the chemotaxis sensitivity $\chi$ and the speed $c$ of the shifting habitat as well as other parameters in  \eqref{Keller-Segel-eq0} on the  persistence and spreading speeds  of the species.

 In the rest of this introduction, we introduce some standing notations and state the main results on the persistence and  spreading speeds  of \eqref{Keller-Segel-eq0}.

\subsection{Notations and statements of  the  main results.}

In order to state our main results, we first introduce some notations and definitions. Let
$$
C^b_{\rm unif}(\R)=\{ u\in C(\R) \, |\, \ u\ \text{is uniformly continuous and bounded on $\R$}\}.
$$
For every $u\in C^b_{\rm unif}(\R)$, we let $\|u\|_{\infty}:=\sup_{x\in\R}|u(x)|$. For each given  $u_0\in C^b_{\rm unif}(\R)$  with $u_0(x)\geq 0$, we denote by $(u(t,x;u_0),v(t,x;u_0))$ the classical solution of \eqref{Keller-Segel-eq0} satisfying $u(0,x;u_0)=u_0(x)$ for every $x\in\R$. Note that, by the comparison principle for parabolic equations, for every  nonnegative initial function $u_0\in C^b_{\rm unif}(\R)$, it always holds that $u(t,x;u_0)\geq 0$ and $v(t,x;u_0)\geq 0$ whenever $(u(t,x;u_0),v(t,x;u_0))$ is defined. In this work we shall only focus on nonnegative classical solutions of \eqref{Keller-Segel-eq0}  since both functions $u(t,x)$ and $v(t,x)$ represent density functions.

The following proposition states the existence and uniqueness of solutions of \eqref{Keller-Segel-eq0} with nonnegative initial functions.

\begin{prop}\label{existence-uniqueness}
For every nonnegative initial function $u_0\in C^b_{\rm unif}(\R)$ and $c\in\R$, {and for any bounded H\"older continuous function $r(x)$}, there is a unique maximal time $T_{max}$, such that $(u(t,x;u_0),v(t,x;u_0))$ is defined for every $x\in\R$ and $0\le t<T_{\max}$. Moreover if $\chi\mu<b$ then $T_{max}=\infty$ and the solution is globally bounded.
\end{prop}

The above proposition can be proved by similar arguments as those in (\cite[Theorem 1.1 and Theorem 1.5]{SaSh1}).

Throughout this paper, we assume that $r(x)$ is as in {\bf Case 1} or {\bf Case 2}. We put
\begin{equation}
\label{r-c-star}
r_*=\inf_{x\in\R} r(x),\quad
r^*=\sup_{x\in\R} r(x),\quad c^*=2\sqrt {r^*}.
\end{equation}
Note that, in {\bf Case 1}, $r_*=r(-\infty)$ and
$r^*=r(+\infty)$, and in {\bf Case 2},
$r_*=\min\{ r(-\infty), r(\infty)\}$ and
$r^*=\max_{x\in\R} r(x)$.

Let $\lambda_L(r(\cdot))$ be the principal eigenvalue of the eigenvalue problem
\begin{equation}
\label{ev-eq0}
\begin{cases}
\phi_{xx}+c\phi_x+ r(x) \phi=\lambda \phi,\quad -L<x<L\cr
\phi(-L)=\phi(L)=0.
\end{cases}
\end{equation}
Note that $\lambda_L(r(\cdot))$ is increasing as $L$ increases {(See \cite[Proposition 4.2]{BeHaRo}).
We also have $\lambda_L(r_1(\cdot))\le \lambda_L(r_2(\cdot))$ if $r_1(\cdot)\le r_2(\cdot)$. In particular,
\begin{equation}
\label{new-new-eq-1}
\lambda_L(r(\cdot))\le \lambda_L(r^*)<r^*.
\end{equation}
Let $\lambda_\infty(r(\cdot))=\lim_{L\to\infty}\lambda_L(r(\cdot))$.
By \eqref{new-new-eq-1},
\begin{equation}
\label{new-new-eq0}
\lambda_\infty(r(\cdot))\le r^*.
\end{equation}
}

For convenience, we make the following standing assumption.

\medskip

\noindent {\bf (H1)} $b>\chi\mu$ and  $b\ge \big(1+\frac{1}{2}\frac{(\sqrt{r^*}-\sqrt{\nu})_+}{(\sqrt{r^*}+\sqrt{\nu})}\big)\chi\mu $.

\medskip
\smallskip

{Note that $b\ge \big(1+\frac{1}{2}\frac{(\sqrt{r^*}-\sqrt{\nu})_+}{(\sqrt{r^*}+\sqrt{\nu})}\big)\chi\mu $ implies $b\ge \chi\mu$;  that $b> \chi\mu$ and $\nu \geq r^{*}$  imply     $b\ge \big(1+\frac{1}{2}\frac{(\sqrt{r^*}-\sqrt{\nu})_+}{(\sqrt{r^*}+\sqrt{\nu})}\big)\chi\mu $ (and hence {\bf (H1)}); and that $b\geq \frac{3}{2} \chi\mu$ implies $b\ge \big(1+\frac{1}{2}\frac{(\sqrt{r^*}-\sqrt{\nu})_+}{(\sqrt{r^*}+\sqrt{\nu})}\big)\chi\mu $ (and hence {\bf (H1)}).
 Biologically, $\nu \geq r^{*}$ means that the degradation rate of the chemo-attractant is greater than or equal to the supremum of the intrinsic growth rate of the mobile species over the whole space, and the condition  $b\geq \big(1+\frac{1}{2}\frac{(\sqrt{r^*}-\sqrt{\nu})_+}{(\sqrt{r^*}+\sqrt{\nu})}\big)\chi\mu$  indicates that the chemotaxis sensitivity  is small relative to the self-limitation rate of the mobile species.}

\smallskip

We now state the main results on the persistence and spreading speeds of the species.

\begin{tm}\label{compact-support}
Suppose that $r(x)$ is as in {\bf Case 1}, {\bf (H1)} holds,
and    $u_0(x)$ is nonnegative, bounded and has a nonempty compact support.

\begin{itemize}
\item[(1)]
If  $c>c^*$,
then $$\lim_{t\to\infty}u(t,x;u_0)=0$$ uniformly for $x\in\R$.

\item[(2)]
If ${-c^*\leq}c<c^*$, then for any $0<\varepsilon<\frac{c^*-c}{2}$, there hold
$$
\lim_{t\to\infty}\sup_{x\le (c-\varepsilon)t}u(t,x;u_0)=\lim_{t\to\infty}\sup_{x\ge (c^*+\varepsilon)t}u(t,x;u_0)=0,
$$
and
$$
\liminf_{t\to\infty}\inf_{(c+\varepsilon)t\le x\le (c^*-\varepsilon)t}u(t,x;u_0)>0.
$$
 Moreover, if $b>2\chi\mu$, then
$$
\lim_{t\to\infty} \sup_{(c+\varepsilon)t\le x\le (c^*-\varepsilon)t}|u(t,x;u_0)-\frac{r^*}{b}|=0.
$$
{
\item[(3)]
 If $c<-c^*$, then for any $0<\varepsilon<c^*$, there hold
 $$
\lim_{t\to\infty}\sup_{x\le (-c^*-\varepsilon)t}u(t,x;u_0)=\lim_{t\to\infty}\sup_{x\ge (c^*+\varepsilon)t}u(t,x;u_0)=0,
$$
and
$$
\liminf_{t\to\infty}\inf_{(-c^*+\varepsilon)t\le x\le (c^*-\varepsilon)t}u(t,x;u_0)>0.
$$
 Moreover, if $b>2\chi\mu$, then
$$
\lim_{t\to\infty} \sup_{(-c^*+\varepsilon)t\le x\le (c^*-\varepsilon)t}|u(t,x;u_0)-\frac{r^*}{b}|=0.
$$
}
\end{itemize}
\end{tm}

\begin{tm}\label{right-support}
Suppose that $r(x)$ is as in {\bf Cases 1},  {\bf (H1)} holds,
and $u_0(x)$ is nonnegative, bounded, and $u_0(x)=0$ for $x\ll  -1$ and $\liminf_{x\to\infty}u_0(x)>0$ .
\begin{itemize}
\item[(1)]
{ If $c\geq-c^*$}, then for any $\varepsilon>0$, there hold
$$\lim_{t\to\infty}\sup_{x\leq(c-\varepsilon)t}u(t,x;u_0)=0,$$
and
$$\liminf_{t\to\infty}\inf_{x\ge  (c+\varepsilon)t}u(t,x;u_0)>0.$$
 Moreover, if $b>2\chi\mu$, then
$$
\lim_{t\to\infty} \sup_{x\ge (c+\varepsilon)t}|u(t,x;u_0)-\frac{r^*}{b}|=0.
$$
{
\item[(2)]
If $c<-c^*$, then for any $\varepsilon>0$, there hold
$$\lim_{t\to\infty}\sup_{x\leq(-c^*-\varepsilon)t}u(t,x;u_0)=0,$$
and
$$\liminf_{t\to\infty}\inf_{x\ge  (-c^*+\varepsilon)t}u(t,x;u_0)>0.$$
 Moreover, if $b>2\chi\mu$, then
$$
\lim_{t\to\infty} \sup_{x\ge (-c^*+\varepsilon)t}|u(t,x;u_0)-\frac{r^*}{b}|=0.
$$
}
\end{itemize}
\end{tm}

\begin{tm}
\label{persistence-extinction-thm1}
Suppose that $r(x)$ is as in {\bf Case 2}, {\bf (H1)} holds, and  $u_0(x)$ is nonnegative, bounded and has a nonempty compact support.

\begin{itemize}
\item[(1)]
If $|c|>c^*$,  then
$$
\lim_{t\to\infty} u(t,x;u_0)=0
$$
uniformly in $x\in\R$.

\item[(2)] If $\lambda_\infty(r(\cdot))<0$ and $\nu\ge \nu^*:={\frac{(\sqrt{8r^*+c^2}+|c|)^2}{4}  }$, then
$$
\lim_{t\to\infty} u(t,x;u_0)=0
$$
uniformly in $x\in\R$.

\item[(3)] If $|c|<c^*$, then
$$
\lim_{t\to\infty} \sup_{|x-ct|\ge c^{'}t}u(t,x;u_0)=0\quad \forall\,\, c^{'}>0.
$$
 If, additionally, $\lambda_\infty(r(\cdot))>0$, then
$$
\liminf_{t\to\infty}\inf_{|x-ct|\le L} u(t,x;u_0)>0\quad \forall\,\, L>0.
$$

\end{itemize}
\end{tm}

We conclude the introduction with the following remarks.

\begin{rk}
\begin{itemize}

\item[(1)] Suppose that $r(x)$ is as in {\bf Case 1} and the species initially lives in a region with ${-M\le} x\le M$ for some $M\in\R$.  Theorem \ref{compact-support}(1)  shows that if
    $c>c^*$, then the species will become extinct in the habitat.  Theorem \ref{compact-support}(2) shows that if ${-c^*\leq}c<c^*$, then the species will persist and spread along the shifting habitat into larger and larger region
at the asymptotic spreading speed $c^*$.Theorem \ref{compact-support}(3) shows that if $c<-c^*$, then the species will persist and spread at the asymptotic spreading speed $c^*$. When $\chi=0$, {\bf (H1)} becomes $b>0$. Hence Theorem \ref{compact-support}(1)  recovers
\cite[Theorem 2.1]{LBSF}, and   Theorem \ref{compact-support}(2) recovers
\cite[Theorem 2.2]{LBSF}.

\item[(2)] Suppose that  $r(x)$ is as in {\bf Case 1} and the species initially lives in a region with $x\ge M$ for some $M\in\R$. Theorem \ref{right-support}{(1)} shows that
for any given { $c\ge -c^*$}, the species will persist and spread along the shifting habitat at the asymptotic spreading speed $c$. Theorem \ref{right-support}{(2)}shows that
for any given $c< -c^*$, the species will persist and spread  at the asymptotic spreading speed $c^*$.

\item[(3)] It is sufficient to assume that $b>\chi\mu$ for $\lim_{t\to\infty}\sup_{x\le (c-\varepsilon)t} u(t,x;u_0)=0$ in Theorem \ref{compact-support}(2). It is not necessary to assume that $b\ge \big(1+\frac{1}{2}\frac{(\sqrt{r^*}-\sqrt{\nu})_+}{(\sqrt{r^*}+\sqrt{\nu})}\big)\chi\mu$ in Theorem \ref{compact-support}(2)  for proving $\liminf_{t\to\infty}\inf_{(c+\varepsilon)t\le x\le (c^*-\varepsilon)t}u(t,x;u_0)>0$ and  in  Theorem \ref{compact-support}(3) for proving  $\liminf_{t\to\infty}\inf_{(-c^*+\varepsilon)t\le x\le (c^*-\varepsilon)t}u(t,x;u_0)>0$. {In Theorem \ref{right-support}, the condition
$b\ge \big(1+\frac{1}{2}\frac{(\sqrt{r^*}-\sqrt{\nu})_+}{(\sqrt{r^*}+\sqrt{\nu})}\big)\chi\mu$ is only needed in proving
$\lim_{t\to\infty}\sup_{x\leq(-c^*-\varepsilon)t}u(t,x;u_0)=0$.
}

 \item[(4)]
 {Suppose that $r(x)$ is as in {\bf Case 2} and the species initially lives in a region with ${-M\le} x\le M$ for some $M\in\R$.  Theorem \ref{persistence-extinction-thm1}(1)  shows that if
    $|c|>c^*$, then the species will die out in the long run.   If $\lambda_{\infty}(r(\cdot))<0$ and the degradation rate $\nu$ of the chemo-attractant is greater than or equal to $\nu^*={\frac{(\sqrt{8r^*+c^2}+|c|)^2}{4}  }$, then the species will also die out in the long run. If $\lambda_{\infty}(r(\cdot))>0$, then the species will persist surrounding the good habitat. When $\chi=0$, {\bf (H1)} becomes $b>0$. Hence Theorem \ref{persistence-extinction-thm1}(2) and (3) recovers
\cite[Theorem 4.11]{BDNZ}. {The assumption $\nu\geq \nu^*$ indicates that the degradation rate of the chemo-attractant is large relative to the speed of the shifting environment.}

\item[(5)] {For the case $\lambda_{\infty}(r(\cdot))=0$, it is still open whether the species will persist or die out in the habitat. }

\item[(6)] In Theorem \ref{persistence-extinction-thm1}, the condition
$b\ge \big(1+\frac{1}{2}\frac{(\sqrt{r^*}-\sqrt{\nu})_+}{(\sqrt{r^*}+\sqrt{\nu})}\big)\chi\mu$ is only needed in (1).}

   \item[(7)] It is not easy to prove the persistence and spreading speeds of solutions of \eqref{Keller-Segel-eq0}. Several new techniques are developed to prove the results stated in the above theorems. These techniques can also be applied to the case $\chi=0$.

{
\item[(8)] Consider the following  attraction Keller-Segel chemotaxis models in higher dimensional shifting environments,
\begin{equation}\label{Keller-Segel-eq-1}
\begin{cases}
u_t=\Delta u-\nabla \cdot (\chi u \nabla v)+u(r(x\cdot \xi-ct)-bu),\quad x\in\R^n\cr
0 =\Delta v-  \nu v +\mu u,\quad x\in\R^n,
\end{cases}
\end{equation}
where $\xi\in\R^n$ is a unit vector.
   We believe that the results stated in Theorems 1.1-1.3 can be extended to \eqref{Keller-Segel-eq-1}.
    For example, for given $u_0\in C_{\rm unif}^b(\R^n)$ with $u_0\ge 0$ and $\{x\,: \, |x\cdot\xi|<r\}\subset {\rm supp}(u_0)\subset \{x\,:\, |x\cdot\xi|<R\}$ for some $0<r<R$, we believe that the statements in Theorem hold with $x\le \tilde c t$ (resp. $x\ge \tilde c t$) being replaced by $x\cdot\xi\le \tilde c t$ (resp. $x\cdot\xi\ge \tilde c t$), where $\tilde c=c-\epsilon$,  $c^*-\epsilon$, or $-c^*-\epsilon$
    (resp. $c+\epsilon$, $c^*+\epsilon$, or $-c^*+\epsilon$).    But, due to the lack of comparison principle for chemotaxis models,
    more new techniques may need to be developed to prove such results. We plan to study the extension of the results in the current paper to
    higher space dimension case somewhere else.
}
\end{itemize}
\end{rk}

The rest of the paper is organized as follows. In section 2, we present some preliminary lemmas to be used in the proofs of the main results.
 In section 3, we study the vanishing and spreading of solutions of  \eqref{Keller-Segel-eq0} with $r(x)$ being as in {\bf Case 1} and
  prove Theorems \ref{compact-support} and \ref{right-support}. In section 4, we study the vanishing and spreading of solutions of  \eqref{Keller-Segel-eq0} with $r(x)$ being as in {\bf Case 2} and prove Theorem \ref{persistence-extinction-thm1}.

\section{Preliminary lemmas}
In this section, we present some preliminary lemmas to be used in the proofs of the main theorems in later sections.

Note that, by the second equation in \eqref{Keller-Segel-eq0},
$$
\Delta v=\nu v-\mu u.
$$
Hence the first equation in \eqref{Keller-Segel-eq0} can be written as
$$
u_t=\Delta u-\chi \nabla u\cdot \nabla v+ u\big( r(x-ct)-\chi \nu v-(b-\chi\mu)u\big),\quad x\in\R.
$$
By the comparison principle for parabolic equations, if $b>\chi\mu$, then for any $u_0\in C_{\rm unif}^b (\R)$ with $u_0\ge 0$,
$$
0\le u(t,x;u_0)\le { \max}\{\|u_0\|_\infty, \frac{r^*}{b-\chi\mu}\}\quad \forall\,\ t\ge 0,\,\, x\in\R.
$$

\begin{lem}
\label{new-lm2}
Assume $b>\chi\mu$.
For every $R\gg 1$, there are $C_R\gg 1$ and $\varepsilon_R>0$ such that for any $u_0\in C_{\rm unif}^b(\R)$ with $u_0\ge 0$, any
$x_0\in \R$, and any $t\ge 0$,  we have
 \begin{equation}\label{nnew-eq3}
\|\chi v_x(t,\cdot;u_0)\|_{L^{\infty}(B_{\frac{R}{2}}(x_0))} + \|\chi\nu v(t,\cdot;u_0)\|_{L^{\infty}(B_{\frac{R}{2}}(x_0))}\leq C_{R}\|u(t,\cdot;u_0)\|_{L^{\infty}(B_{R}(x_0))}+\varepsilon_R M
\end{equation}
with $\lim_{R\to\infty}\varepsilon_R=0$, where $M:=\max\{\|u_0\|_{\infty},\frac{r^*}{b-\chi\mu}\}$.
\end{lem}

\begin{proof} It follows from the arguments of  \cite[Lemma 2.5]{SaShXu}.
\end{proof}



\begin{lem}
\label{new-lm3}
{ Assume $b>\chi\mu$.} For every $p>1$, $t_0>0$, $s_0\geq 0$, $R>0$, and  $u_0\in C_{\rm unif}^b(\R)$ with $u_0\ge 0$, there is $C_{t_0,s_0,R,M,p}$ such that if $s\in[0,s_0], \ t\geq t_0,  \ |x-y|\leq R$, then
\begin{equation}\label{nnew-eq4}
u(t,x;u_0)\le C_{t_0,s_0,R,M,p}[u(t+s,y;u_0)]^{\frac{1}{p}}(M+1),
\end{equation}
where   $M:=\max\{\|u_0\|_{\infty},\frac{r^*}{b-\chi\mu}\}$.
\end{lem}

\begin{proof}
The lemma can be proved by slightly  modified  arguments  of \cite[Lemma 2.2]{FhCh}.

In fact, first, fix $t_0>0, s_0\geq 0, R>0, p\in (1,\infty)$ and $t\geq t_0>0$. Let
$\delta=\min\{\frac{t_0}{2}, 1\}$ and
$$
A=bM+\sup_{t>0}\big(\|\left|\chi\nabla v(t,\cdot;u_0)\right|\|_{\infty}+\|\nabla\cdot\chi\nabla v(t,\cdot;u_0)\|_{\infty}\big).
$$
Note that $\|u(t, \cdot; u_0)\|_{\infty}\leq M$ for all $t\ge 0$.
 Let
$ \bar u(s',x)$ be the solution of
\begin{equation}
\begin{cases}
\bar u_{s'}+\chi\nabla v(s', \cdot; u_0)\cdot \nabla \bar u=\Delta \bar u \quad {\rm in}\, (t-\delta, \infty)\times\R \cr
\bar u(t-\delta, \cdot)=u(t-\delta, \cdot; u_0)   \quad  {\rm in}\,\  \R.
\end{cases}
\end{equation}
 By the comparison principle for parabolic equations, we have
  $$0\leq\bar u(s',x)\leq \|u(t-\delta,\cdot;u_0)\|_{\infty}\quad {\rm in}\quad (t-\delta, \infty)\times \R.
  $$

  Next,  let
$$
u^-(s',x)=\min\{1,\|u(t-\delta,\cdot;u_0)\|^{-1}_{\infty}\}e^{(r_*-A)(s'-t+\delta)}\bar u(s',x)
$$
 and
$$
u^+(s',x)=e^{(r^*+A)(s'-t+\delta)}\bar u(s',x).
$$
Notice that
\begin{align}
A&>bM+\nabla\cdot\chi\nabla v(s',\cdot;u_0)\cr
&>b\|u(t-\delta,\cdot;u_0)\|_{\infty}+\nabla\cdot\chi\nabla v(s',\cdot;u_0)\cr
&\geq b\bar u(s',x)+\nabla\cdot\chi\nabla v(s',\cdot;u_0)\cr
&\geq bu^-(s',x)+\nabla\cdot\chi\nabla v(s',\cdot;u_0)\,\,\,  \forall\, s'>t-\delta,\,\,  x\in\R.
\end{align}
By a straightforward computation, we have
\begin{align}
&u^-_{s'}+\chi\nabla v(s',\cdot; u_0)\cdot \nabla u^-+u^- \nabla\cdot\chi\nabla v(s',\cdot;u_0)\cr
&=(r_*-A)u^-+\Delta u^-
+u^- \nabla\cdot\chi\nabla v(s',\cdot;u_0)\cr
&\leq r(x-cs')u^-+\Delta u^-+u^- (\nabla\cdot\chi\nabla v(s',\cdot;u_0)-A)\cr
&\leq \Delta u^-+u^-(r(x-cs')-bu^-)\,\,\, \forall\, s'>t-\delta, x\in\R
\end{align}
and
\begin{align}
&u^+_{s'}+\chi\nabla v(s',\cdot; u_0)\cdot \nabla u^++u^+ \nabla\cdot\chi\nabla v(s',\cdot;u_0)\cr
&=( r^*+A)u^++\Delta u^+
+u^+ \nabla\cdot\chi\nabla v(s',\cdot;u_0)\cr
&\geq \Delta u^++u^+(r(x-cs')-bu^+)\,\,\, \forall\, s'>t-\delta, x\in\R.
\end{align}
Hence, $u^-(s',x)$ and $u^+(s',x)$ are, respectively, a sub- and super-solution of the following equation for $(s',x)\in(t-\delta,\infty)\times \R$,
$$
u_{s'}=\Delta u-\chi \nabla v(s',x;u_0)\cdot\nabla u-\chi u \nabla \cdot\nabla v(s',x;u_0)+u(r(x-cs')-bu),\quad x\in\R.
$$
Furthermore,
$$u^-(t-\delta,\cdot)\leq u(t-\delta,\cdot, u_0)=u^+(t-\delta,\cdot).
 $$
 Then by the comparison principle for parabolic equations, we have that, for all $(s', x)\in(t-\delta, \infty)\times \R$,
\begin{align}\label{ineq-1}
\min\{1,\|u(t-\delta,\cdot;u_0)\|^{-1}_{\infty}\}e^{(r_*-A)(s'-t+\delta)}\bar u(s',x)&\leq u(s',x; u_0)\nonumber\\
&\leq e^{(r^*+A)(s'-t+\delta)}\bar u(s',x).
\end{align}

Now by \eqref{ineq-1},
$$
u(t,x; u_0)\le e^{(r^*+A)\delta} \bar u(t,x)\quad \forall\, x\in\R.
$$
By the arguments of \cite[Lemma 2.2]{FhCh},  there is $C_{t_0,s_0,R,M,p}$ such that if $s\in[0,s_0], \ t\geq t_0,  \ |x-y|\leq R$,
then \eqref{nnew-eq4} holds. The lemma thus follows.
\end{proof}

By \eqref{nnew-eq3} and \eqref{nnew-eq4} with $p>1$, $s_0=0$ and $t_0=1$, we have
\begin{align}
\label{E2-1}
\left|\chi v_x(t,x;u_0)\right|+\chi \nu v(t,x;u_0)&\le C_{R,p} \big(u(t,x;u_0)\big)^{\frac{1}{p}}+\varepsilon_R M\nonumber\\
&= C_{R,p} \big(u(t,x;u_0)\big)^{\frac{1}{p}}+\varepsilon_R M \quad \forall\,\, t\ge 1,\,\, x\in\R,
\end{align}
where  $C_{R,p}=C_R\cdot  C_{1,0,R,M,p}\cdot (M+1) (>0)$.


\begin{lem}
\label{lem-001-1}
 Assume that $b>\chi\mu$. Let $c_{\kappa}=\frac{\kappa^2+r^*}{\kappa}$ with $ 0<\kappa\leq \sqrt{r^*}$ satisfying
$$\frac{(\kappa-\sqrt{\nu})_{+}}{(\kappa +\sqrt{\nu})}\leq \frac{2(b-\chi\mu)}{\chi\mu}.$$ The following hold.
\begin{enumerate}
\item[(i)]
 For any $u_0\ge 0$ with  nonempty compact support
 and  any $M\gg \frac{r^*}{b-\chi\mu}$  satisfying
\begin{equation}\label{u_0-cond1}
\max\{u_0(x),u_0(-x)\}\leq U^+(x):=\min\{M,Me^{-\kappa x}\},\quad \forall\ x\in\R,
\end{equation}
there holds
\begin{equation}
\label{new-eq1}
u(t,x;u_0)\leq M e^{-\kappa(|x|-c_{\kappa}t)}, \quad\forall\ x\in\R,\ t\geq 0.
\end{equation}

\item[(ii)]
 For any $u_0\in C^b_{\rm unif}(\R)$, $u_0(x)\geq 0 $,   and  any  $M\gg \frac{r^*}{b-\chi\mu}$  satisfying
\begin{equation}\label{u_0-cond3}
u_0(x)\leq U^+(x):=\min\{M,Me^{\kappa x}\},\quad \forall\ x\in\R,
\end{equation}
there holds
\begin{equation}
\label{new-eq1-3}
u(t,x;u_0)\leq M e^{\kappa(x+c_{\kappa}t)}, \quad\forall\ x\in\R,\ t\geq 0.
\end{equation}
\end{enumerate}
\end{lem}

\begin{proof}
It follows from the arguments of \cite[Lemma 2.3]{SaShXu}.
\end{proof}

Consider
\begin{equation}\label{Le2.6-eq1}
\begin{cases}
u_t=u_{xx}-\chi(uv_x)_x +u(r^*-bu),\quad x\in\R\cr
0=v_{xx}- \nu v+\mu u,\quad x\in\R.
\end{cases}
\end{equation}

\begin{lem}
\label{lem-001-4}
Assume that $b>2\chi\mu$. Then for any $u_0\in C_{\rm unif}^b(\R)$ with $\inf_{x\in\R}u_0(x)>0$,
$$
\lim_{t\to\infty} \|u(t,\cdot;u_0)-\frac{r^*}{b}\|_\infty=0,
$$
where $(u(t,x;u_0),v(t,x;u_0))$ is the solution of \eqref{Le2.6-eq1} with $u(0,x;u_0)=u_0(x)$.
\end{lem}

\begin{proof}
It follows from \cite[Theorem 1.8]{SaSh1}.
\end{proof}

\section{Persistence and spreading speeds in Case 1}

In this section, we study persistence and spreading speeds of \eqref{Keller-Segel-eq0} with $r(x)$ being as in {\bf Case 1}  and prove Theorems \ref{compact-support} and \ref{right-support}.  Throughout this section, we assume that {\bf (H1)} holds and $r(x)$ is as in {\bf Case 1}.

We first prove two lemmas.

Observe that for any given $\bar c$, let $u(t,x)=\tilde u(t,x-\bar c t)$ and $v(t,x)=\tilde v(t,x-\bar c t)$. Then \eqref{Keller-Segel-eq0} becomes
\begin{equation}\label{new-Keller-Segel-eq}
\begin{cases}
\tilde u_t=\Delta\tilde u+\bar c \tilde u_x-\nabla \cdot (\chi\tilde  u \nabla\tilde  v)+\tilde u(r(x-(c-\bar c)t)-b\tilde u),\quad x\in\R\cr
0 =\Delta \tilde v-  \nu \tilde v +\mu \tilde u,\quad x\in\R.
\end{cases}
\end{equation}
In the following,  $(u(t,x;u_0),v(t,x;u_0))$ denotes the solution of \eqref{Keller-Segel-eq0} with
$u(0,x;u_0)=u_0(x)$, and $(\tilde u(t,x;u_0),\tilde v(t,x;u_0))$ denotes the solution of \eqref{new-Keller-Segel-eq} with
$\tilde u(0,x;u_0)=u_0(x)$.

For any given $0<\epsilon<2\sqrt{r^*}$, fix $\bar r<r^*$  such that
\begin{equation}
\label{bar-r-eq}
4\bar r -\bar c^2\ge \epsilon \sqrt{r^*}\quad \forall\,\,   -2\sqrt{r^*}+\epsilon \le  \bar c  \le  2\sqrt {r^*}-\epsilon.
\end{equation}
 Let
\begin{equation}
\label{l-eq}
l=\frac{2\pi}{(\epsilon\sqrt{r^*})^{\frac{1}{2}}}
\end{equation}
and
\begin{equation}
\label{lambda-eq}
\lambda(\bar c,\bar r)=\frac{4\bar r-\bar c^2-\frac{\pi^2}{l^2}}{4}.
\end{equation}
 Then $\lambda(\bar c,\bar r)\ge  \frac{3\epsilon\sqrt {r^*}}{16}>0$  for any $ -2\sqrt{r^*}+\epsilon  \le \bar c \le   2\sqrt {r^*}-\epsilon$.

\begin{lem}
\label{lem-001-2}
For given $0<\epsilon<2\sqrt{r^*}$, {let $\bar r$ and $l$ be as in \eqref{bar-r-eq} and
\eqref{l-eq}. Then
for any $ -2\sqrt{r^*}+\epsilon \leq \bar c \leq  2\sqrt {r^*}-\epsilon$, $\lambda(\bar c,\bar r)$ which is defined as in \eqref{lambda-eq}} is the principal eigenvalue of
\begin{equation}
\label{ev-eq1}
\begin{cases}
\phi_{xx}+\bar c\phi_x+\bar r \phi=\lambda \phi,\quad -l<x<l\cr
\phi(-l)=\phi(l)=0,
\end{cases}
\end{equation}
and $\phi(x;\bar c,\bar r)=e^{-\frac{\bar c}{2}x}{ \cos\frac{\pi}{2l}x}$ is a corresponding positive eigenfunction.
\end{lem}

\begin{proof}
It follows from direct calculations.
\end{proof}

Consider
\begin{equation}
\label{aux-new-eq01}
\begin{cases}
u_t=u_{xx}+\bar c u_x-A(t,x)u_x+u (r^* -2 \varepsilon_R M- C_{R,p} u^{\frac{1}{p}}-(\bar b-\chi\mu)u), \quad -l<x<l\cr
u(t,-l)=u(t,l)=0,
\end{cases}
\end{equation}
where $R\gg 1$ is such that $r^*-2 \varepsilon_R M>\bar r$,  and $A(t,x)$ is globally H\"older continuous in $t\in\R$ and
$x\in [-l,l]$ with H\"older exponent $0<\alpha<1$ and $\|A(\cdot,\cdot)\|_\infty<\infty$.

\begin{lem}
\label{lem-001-3} For given $0<\epsilon<2\sqrt{r^*}$, {let $\bar r$ and $l$} be as in \eqref{bar-r-eq} and
\eqref{l-eq}.
There is $\eta>0$ such that for any $A(\cdot,\cdot)$ with $\|A(\cdot,\cdot)\|_\infty<\eta$,  any $\bar c\in {[ -2\sqrt{r^*}+\epsilon} ,2\sqrt{r^*}-\epsilon { ]}$, and
any $\bar b \in (\chi\mu,\infty)$,
\eqref{aux-new-eq01} has a unique positive bounded entire solution
$u(t,x;\bar c,\bar b, A)$ satisfying that
 \begin{equation}
 \label{aux-new-new-eq0}
 \inf_{-l+\delta\le x\le l-\delta,t\in\R, \bar c\in {[ -2\sqrt{r^*}+\epsilon} ,2\sqrt{r^*}-\epsilon { ]}} u(t,x;\bar c,\bar b,A)>0\quad \forall \, 0<\delta <l.
 \end{equation}
\end{lem}

\begin{proof}
First, consider
\begin{equation}
\label{aux-new-eq02}
\begin{cases}
\tilde u_t=\tilde u_{xx}+\bar c\tilde  u_x-A(t,x)\tilde u_x+\bar r \tilde u , \quad  -l<x<l\cr
\tilde u(t,-l)=\tilde u(t,l)=0.
\end{cases}
\end{equation}
For given $t_0\in\R$, let $\tilde u(t,x;t_0,\phi(\cdot;\bar c,\bar r),A)$ be the solution of \eqref{aux-new-eq02}
with $\tilde u(t_0,x;t_0,\phi(\cdot;\bar c,\bar r),A)= \phi(x;\bar c,\bar r)$.
 By \cite[Theorem 3.4.1]{Hen},
\begin{equation}
\label{aux-new-new-eq2}
\lim_{\|A(\cdot,\cdot)\|_{\infty}\to 0} \|\tilde u(t,\cdot;t_0,\phi(\cdot;\bar c,\bar r),A)-e^{\lambda(\bar c,\bar r)(t-t_0)}\phi(\cdot;\bar c,\bar r)\|_{C^1([-l,l])}=0
\end{equation}
uniformly in $t\in [t_0,t_0+1]$, $t_0\in\R$, and  $\bar c\in {[ -2\sqrt{r^*}+\epsilon} ,2\sqrt{r^*}-\epsilon { ]}$.
This implies that there is $\eta>0$ such that for any $A(\cdot,\cdot)$ with $\|A(\cdot,\cdot)\|_\infty<\eta$, and  any
 $\bar c\in {[ -2\sqrt{r^*}+\epsilon} ,2\sqrt{r^*}-\epsilon { ]}$,
\begin{equation}
\label{aux-new-new-eq3}
 \tilde u(t_0+1,x;t_0,\phi(\cdot;\bar c,\bar r),A)\ge  e^{\lambda(\bar c,\bar r) /2}\phi(x;\bar c,\bar r)\,\,\,\forall\,  -l\le x\le l,\,\, t_0\in\R.
\end{equation}

 Next, suppose that $\|A(\cdot,\cdot)\|_\infty<\eta$.
 Let $u(t,x;t_0,\sigma \phi(\cdot;\bar c,\bar r))$ be the solution of \eqref{aux-new-eq01} with
 $u(t_0,x;t_0,\sigma\phi(\cdot;\bar c,\bar r))=\sigma\phi(x;\bar c,\bar r)$.
 Note that
$$
\lim_{\sigma\to 0} \sup_{t\in [t_0,t_0+1], -l\le x\le l}  u(t,x;t_0,\sigma\phi(\cdot;\bar c,\bar r))=0
$$
uniformly in $t_0\in\R$ and in  $\bar c\in {[ -2\sqrt{r^*}+\epsilon} ,2\sqrt{r^*}-\epsilon { ]}$.
Hence there is $\sigma_0>0$ such that for any $0<\sigma\le \sigma_0$, and $\bar c\in {[ -2\sqrt{r^*}+\epsilon} ,2\sqrt{r^*}-\epsilon { ]}$,
$$
r^* - 2  \varepsilon_R M- C_{R,p}  u^{\frac{1}{p}}(t,x;t_0,\sigma\phi)-(\bar b-\chi\mu)u (t,x;t_0,\sigma\phi)>\bar r\,\,
\forall\, t\in [t_0,t_0+1],\,\, -l\le x\le l,\,\, t_0\in\R.
$$
This together with the comparison principle for parabolic equations implies that for $0<\sigma \le \sigma_0$ and $\bar c\in {[ -2\sqrt{r^*}+\epsilon} ,2\sqrt{r^*}-\epsilon { ]}$,
\begin{equation}
\label{aux-new-new-eq1}
u(t,x;t_0,\sigma \phi(\cdot;\bar c,\bar r))\ge \sigma \tilde u(t,x;t_0,\phi,A)\,\,\,\forall\, t\in [t_0,t_0+1],\,\, -l\le x\le l,\,\, t_0\in\R.
\end{equation}
Then by \eqref{aux-new-new-eq3}, we have
\begin{equation}
\label{aux-new-new-eq3-1}
  u(t_0+1,x;t_0,\sigma \phi(\cdot;\bar c,\bar r))  \ge \sigma  e^{\lambda(\bar c,\bar r) /2}\phi(x;\bar c,\bar r)\,\,\,\forall\,  -l\le x\le l,\,\, t_0\in\R
\end{equation}
for any $0<\sigma\le \sigma_0$ and $\bar c\in {[ -2\sqrt{r^*}+\epsilon} ,2\sqrt{r^*}-\epsilon { ]}$.

\medskip

Now, by  \eqref{aux-new-new-eq3-1} and  the comparison principle for parabolic equations, we have
\begin{equation}
\label{aux-new-new-eq4-0}
u(k,x;-n,\sigma\phi(\cdot;\bar c,\bar r))>\sigma\phi(x;\bar c,\bar r) \quad \forall \,\, k\ge -n+1,\,\,  -l<x<l
\end{equation}
and then
\begin{equation}
\label{aux-new-new-eq4}
u(k,x;-(n+1),\sigma\phi(\cdot;\bar c,\bar r))>u(k,x;-n,\sigma\phi(\cdot;\bar c,\bar r))\quad \forall \,\,k\ge -n+1,\,\, -l<x<l.
\end{equation}
Let $u_n(t,x)=u(t,x;-n,\sigma\phi(\cdot;\bar c,\bar r))$.
Then $\lim_{n\to \infty} u_n(t,x)$ exists and
 $u(t,x)=\lim_{n\to\infty}u_n(t,x)$ is a solution of  \eqref{aux-new-eq01}.
 By \eqref{aux-new-new-eq4-0} and \eqref{aux-new-new-eq4}, $u(t,x)$ is a positive bounded entire solution of \eqref{aux-new-eq01} satisfying \eqref{aux-new-new-eq0}.

\medskip

 Finally, we prove that when $\|A(\cdot,\cdot)\|_\infty<\eta$, \eqref{aux-new-eq01} has a unique positive bounded entire solution  satisfying \eqref{aux-new-new-eq0}.
 Suppose that $u(t,x)$, $v(t,x)$ are two positive bounded  entire solutions of \eqref{aux-new-eq01}  satisfying \eqref{aux-new-new-eq0}. By  Hopf's lemma,
 $$
 u_x(t,-l)>0,\,\, u_x(t,l)<0,\,\, v_x(t,-l)>0,\,\, v_x(t,l)<0\quad \,\, \forall\, t\in\R.
 $$
 This implies that for any $t\in\R$, the following set is not empty,
 $$
 \{\gamma>1\,|\, \frac{1}{\gamma} u(t,x)\le v(t,x)\le \gamma u(t,x)\quad \forall\,\, -l<x<l\}.
 $$
 Hence we can define
 $$
 \rho(u(t,\cdot),v(t,\cdot))=\inf\{\ln \gamma \,|\, \frac{1}{\gamma} u(t,x)\le v(t,x)\le \gamma u(t,x)\quad \forall\,\, -l<x<l\}.
 $$
 To prove the uniqueness of positive entire solutions  satisfying \eqref{aux-new-new-eq0}, it then suffices to prove $\rho(u(t,\cdot),v(t,\cdot)) \equiv 0$.

 Fix $t_0\in\R$. Suppose that $\gamma>1$ is such that
 $$
 \frac{1}{\gamma} u(t_0,x)\le v(t_0,x)\le \gamma u(t_0,x)\quad \forall\,\, -l<x<l.
 $$
 By the comparison principle for parabolic equations, we have
 $$
 \frac{1}{\gamma} u(t,x;t_0,u(t_0,\cdot))<v(t,x;t_0,v(t_0,\cdot))<\gamma u(t,x;t_0,u(t_0,\cdot))\quad \forall\,\, t>t_0,\,\, -l<x<l.
 $$
 This together with Hopf's lemma implies that there is $1<\gamma(t)<\gamma$ such that
 $$
 \frac{1}{\gamma(t)} u(t,x;t_0,u(t_0,\cdot))<v(t,x;t_0,v(t_0,\cdot))<\gamma(t) u(t,x;t_0,u(t_0,\cdot))\quad \forall\,\, t>t_0,\,\, -l<x<l.
 $$
 Hence if $\rho(u(t_0,\cdot),v(t_0,\cdot)))\not =0$ for some $t_0\in\R$, then  $\rho(u(t,\cdot),v(t,\cdot))$ is strictly decreasing as $t$ increases.

 Assume that $\rho(u(t,\cdot),v(t,\cdot))\not \equiv 0$. Let $\rho^*=\lim_{t\to -\infty}  \rho(u(t,\cdot),v(t,\cdot))  $.
 Then $\rho^*>0$. Choose a sequence $t_n\to -\infty$. Without loss of generality, we may assume that
 $$
 A(t_n+t,x)\to A^*(t,x),\quad u(t_n+t,x)\to u^*(t,x),\quad v(t_n+t,x)\to v^*(t,x)
 $$
 as $n\to\infty$ uniformly in $x\in [-l,l]$ and locally uniformly in $t\in\R$.
 We then have that $u^*(t,x)$ and $v^*(t,x)$ are positive solutions of  \eqref{aux-new-eq01} with $A(t,x)$ being replaced by $A^*(t,x)$ and
 satisfy \eqref{aux-new-new-eq0}.
 Moreover,
 $$
 \rho(u^*(t,\cdot),v^*(t,\cdot))=\rho^*>0\quad \forall\, t\in\R.
 $$
 But by the arguments in the above, $\rho(u^*(t,\cdot),v^*(t,\cdot))$ is strictly decreasing as $t$ increases, which is a contradiction.

 Therefore, $\rho(u(t,\cdot),v(t,\cdot))\equiv 0$ and $u(t,x)\equiv v(t,x)$.
\end{proof}

Next,  we prove Theorem \ref{compact-support}.

\begin{proof}[Proof of Theorem \ref{compact-support}]

(1) { Suppose $c>c^*=2\sqrt{r^*}$.}
Choose $\bar c$ and  $0<\kappa\le \sqrt{r^*}$ such that
$$\frac{(\kappa-\sqrt{\nu})_{+}}{(\kappa +\sqrt{\nu})}\leq \frac{2(b-\chi\mu)}{\chi\mu}$$
and
$$
c^*=2\sqrt {r^*}<c_\kappa<\bar c<c.
$$
By Lemma \ref{lem-001-1}(i),
$$
\lim_{t\to\infty}\sup_{|x|\ge \bar ct }  u(t,x;u_0)=0.
$$

We claim that $\lim_{t\to\infty}\sup_{x\in\R}  u(t,x;u_0)=0$. For otherwise, there are $\delta_0>0$, $t_n\to\infty$ and $x_n\in (-\bar c t_n, \bar c t_n)$
such that
$$
 u(t_n,x_n;u_0)\ge \delta_0\quad \forall\, n\ge 1.
$$
Let $u_n(t,x)=u(t+t_n,x+x_n;u_0)$ and $v_n(t,x)= v(t+t_n,x+x_n;u_0)$. Note  that
  $x_n-c t_n\to  -\infty$ as $n\to\infty$.
  Without loss of generality, we may assume that there is $(u^*(t,x),v^*(t,x))$ such that
$$
\lim_{n\to\infty}(u_n(t,x),v_n(t,x))=(u^*(t,x),v^*(t,x))
$$
locally uniformly in $(t,x)\in\R\times \R$, and $(u^*(t,x),v^*(t,x))$ satisfies
\begin{equation}\label{new-Keller-Segel-eq1}
\begin{cases}
u_t^*=\Delta u^*-\nabla \cdot (\chi u^* \nabla v^*)+u^*(r(-\infty)-bu^*),\quad t\in\R,\,\, x\in\R\cr
0 =\Delta v^*-  \nu v^* +\mu u^*,\quad t\in\R,\,\, x\in\R.
\end{cases}
\end{equation}
By $r(-\infty)<0$, it can be proved that $u^*(t,x)\equiv 0$, which contradicts to
$$u^*(0,0)=\lim_{n\to\infty}  u(t_n,x_n;u_0)\ge \delta_0.$$

Therefore, $\lim_{t\to\infty} \sup_{x\in\R}  u(t,x;u_0)=0$.

\medskip

(2) { Suppose $-c^*\leq c<c^*=2\sqrt{r^*}$.} We first prove
  $$
\lim_{t\to\infty}\sup_{x\le (c-\varepsilon)t}u(t,x;u_0)=0.
$$
Assume that the result does not hold. Then there are {  constants $\delta_0>0$, $\varepsilon_0>0$, and a sequence
$\{(t_n, x_n)\}_{n\in\N}$,  $t_n\to\infty$, $x_n\in (-\infty,(c-\varepsilon_0)t_n]$
such that }
$$
u(t_n,x_n;u_0)\ge \delta_0\quad \forall\, n\ge 1.
$$
Let $u_n(t,x)=u(t+t_n,x+x_n;u_0)$ and $v_n(t,x)=v(t+t_n,x+x_n;u_0)$.    Note  that
  $x_n-c t_n\to  -\infty$ as $n\to\infty$.
Similarly, without loss of generality, we may assume that there is $(u^*(t,x), v^*(t,x))$ such that
$$
\lim_{n\to\infty}(u_n(t,x),v_n(t,x))=(u^*(t,x),v^*(t,x))
$$
locally uniformly in $(t,x)\in\R\times \R$, and $(u^*(t,x),v^*(t,x))$ satisfies \eqref{new-Keller-Segel-eq1}.
Again, by $r(-\infty)<0$, it can be proved that $u^*(t,x)\equiv 0$, which contradicts to
$$u^*(0,0)=\lim_{n\to\infty} u(t_n,x_n;u_0)\ge \delta_0.$$
Therefore, $\lim_{t\to\infty} \sup_{x\le (c-\varepsilon)t} u(t,x;u_0)=0$.

Next, we prove
$$
\lim_{t\to\infty}\sup_{x\ge (c^*+\varepsilon)t}u(t,x;u_0)=0.
$$
For any $\varepsilon>0$, let $\kappa=\sqrt{r^*}$, then $c_{\kappa}=2\sqrt{r^*}<c^*+\varepsilon$.
Since $0<\chi\mu<b$, $\big(1+\frac{1}{2}\frac{(\sqrt{r^*}-\sqrt{\nu})_+}{(\sqrt{r^*}+\sqrt{\nu})}\big)\chi\mu { \leq} b$, by Lemma \ref{lem-001-1}(i),
$$
u(t,x;u_0)\le Me^{-\kappa(|x|-c_{\kappa}t)},
$$
and
$$
\sup_{x\ge (c^*+\varepsilon)t} u(t,x;u_0)\le Me^{-\kappa(c^*+\varepsilon-c_\kappa)t}\to0 \quad {\rm as} \quad  t\to\infty.
$$
Therefore, $\lim_{t\to\infty}\sup_{x\ge (c^*+\varepsilon)t}u(t,x;u_0)=0$.

\smallskip

We now prove
\begin{equation}\label{Tm1.1.2liminf}
\liminf_{t\to\infty}\inf_{(c+\varepsilon)t\le x\le (c^*-\varepsilon)t} u(t,x;u_0)>0.
\end{equation}

To this end, for any $0<\varepsilon<\frac{c^*-c}{2}$,  {let $\bar r$ and $l$ be as in \eqref{bar-r-eq} }and
\eqref{l-eq}.
Fix a $\bar c$ satisfying ${ -c^*+\varepsilon\leq}  c+\varepsilon\le \bar c\le c^*-\varepsilon$  and set $M=\max\{\|u_0\|_{\infty},\frac{r^*}{b-\chi\mu}\}$. {Let $\lambda(\bar c, \bar r)$ be as in \eqref{lambda-eq}.} By \eqref{E2-1},  for any  $R\gg 1, { p>1} $, $(\tilde u(t,x;u_0),\tilde v(t,x;u_0))$ satisfies
\begin{equation}\label{AA-eq1}
\tilde u_t\geq \tilde u_{xx}+\bar c \tilde u_x-\chi \tilde v_x\tilde u_x+\tilde u(r(x-(c-\bar c)t)-\varepsilon_R M- C_{R,p} \tilde u^{\frac{1}{p}}-(b-\chi\mu)\tilde u), \quad t\geq 1,\ x\in\R.
\end{equation}
Let $p=2$ and $\eta$ be as in
Lemma \ref{lem-001-3}. Choose  $R\gg 1$  such that $
\varepsilon_{R}M<\frac{\eta}{4}$,
\begin{equation}\label{AA-eq2-1}
|\chi \tilde v_x|\leq C_{R}\sqrt {\tilde u(t,x;u_0)}+\frac{\eta}{4}, \quad t\geq 1,\ x\in\R.
\end{equation}
Define
$$
A(t,x)=\begin{cases}
\frac{\chi \tilde v_x(1,x;u_0)}{\max\{1, |\chi \tilde v_x(1,x;u_0)|\eta^{-1}\}}, \,\, &{\rm if} \,\, t< 1, \ x\in\R \cr
\frac{\chi \tilde v_x(t,x;u_0)}{\max\{1, |\chi \tilde v_x(t,x;u_0)|\eta^{-1}\}},     \,\, &{\rm if}\,\, t\geq 1,\ x\in\R.
\end{cases}
$$
 Note that $\tilde v_x(t,x;u_0)$ is globally H\"older continuous in $t\geq 1$ and $x\in \R$.  We then have that $A(t,x)$ is globally H\"older continuous in $t\in\R$ and
$x\in \R$. It is clear that  $\|A(\cdot,\cdot)\|_\infty<\eta$.

 Let $T>1$ be such that
$r(x-(c-\bar c)t)\ge r^*-\varepsilon_R M$ for $x\ge -l$ and $t\ge T$.
Choose $\bar b>b$ and also $\bar b \gg 1$ such that
 \begin{equation}\label{AA-eq2-2}
 C_{R}\sqrt{\frac{r^*}{\bar b-\chi\mu}}+\frac{r^*}{\bar b-\chi\mu}<\frac{\eta}{4},
 \end{equation}
 and
 $u(T,x;\bar c,\bar b, A)<\frac{r^*}{\bar b-\chi\mu}< \tilde u(T,x;u_0)$  for $-l\le x\le l$ and $\bar c\in [c+\varepsilon,c^*-\varepsilon]$, where
$u(t,x;\bar c,\bar b, A)$ is
 the unique  positive bounded entire solution of \eqref{aux-new-eq01}.

 \medskip

 Fix such $\bar b$. We first claim that for any $c+\varepsilon \le \bar c\le c^*-\varepsilon$,
\begin{equation}\label{AA-eq3}
\tilde u(t,x;u_0)\ge u(t,x;\bar c,\bar b, A)\quad \forall\, t\ge T,\,\, -l\le x\le l.
\end{equation}

Suppose, by contradiction that \eqref{AA-eq3} does not hold. Then there { are $c+\varepsilon \le \bar c\le c^*-\varepsilon$ and } $t_{\rm inf}\in[T,\infty)$ satisfying
$$
t_{\rm inf}:=\inf\{t\in(T,\infty)\, |\,  \exists  x_t\in\R, {\rm satisfying }\quad u(t, x_t; \bar c,\bar b, A)>\tilde u(t,x_t; u_0), |x_t|\leq l\}.
$$
Note that
$$
 u(T,x;\bar c,\bar b, A)<\tilde u(T,x;u_0) \quad \forall -l\le x\le l
$$
Hence
$$t_{\rm inf}>T. $$
Moreover, note that $u(t,-l;\bar c,\bar b, A)=u(t,l;\bar c,\bar b, A)=0$ for any $t\in\R$, there is $x_{\rm inf}\in\R$ such that $|x_{\rm inf}|<l$,
\begin{equation}\label{AA-eq7}
\frac{r^*}{\bar b-\chi\mu}>u(t_{\rm inf},x_{\rm inf};\bar c,\bar b, A)=\tilde u(t_{\rm inf},x_{\rm inf};u_0),
\end{equation}
and
$$
u(t,x;\bar c,\bar b, A)<\tilde u(t,x;u_0),\quad\ |x|\leq l,\quad T\le t<t_{\rm inf}.
$$
 Hence there is $0<\delta\ll 1$ such that  $[t_{\rm inf}-\delta,t_{\rm inf}]\times [x_{\rm inf}-\delta,x_{\rm inf}+\delta]\subset \{(t,y)\,|\,|y|<l\}$ and
$$
A(t,x)=\chi \tilde v_x(t,x;u_0)
,\quad \forall\ t_{\rm inf}-\delta\le t\leq t_{\rm inf},\quad x_{\rm inf}-\delta\leq x\leq x_{\rm inf}+\delta.$$
Note that
$$
u(t_{\rm inf}-\delta,x;\bar c,\bar b, A)<\tilde u(t_{\rm inf}-\delta,x;u_0) \quad \forall\, x\in[x_{\rm inf}-\delta,x_{\rm inf}+\delta]
$$
and
$$
u(t,x_{\rm inf}\pm\delta;\bar c,\bar b, A)\leq \tilde u(t,x_{\rm inf}\pm\delta; u_0) \quad \forall\, t_{\rm inf}-\delta\leq t\leq t_{\rm inf}.
$$
Thus, by the comparison principle for parabolic equations, we have
$$
u(t,x;\bar c,\bar b, A)<\tilde u(t,x; u_0) \quad\forall\, t_{\rm inf}-\delta< t\leq t_{\rm inf}, \quad x_{\rm inf}-\delta< x< x_{\rm inf}+\delta.
$$
In particular,
$$
u(t_{\rm inf},x_{\rm inf};\bar c,\bar b, A)<\tilde u(t_{\rm inf},x_{\rm inf};u_0).
$$
Which contradicts to \eqref{AA-eq7}.

\medskip

By Lemma \ref{lem-001-3},
$$
\inf_{-l+\delta\le x\le l-\delta,c+\varepsilon\le \bar c\le c^*-\varepsilon, t\ge 1}u(t,x;\bar c,\bar b, A)>0, \quad \forall\,\, 0<\delta <l.
$$
This together with \eqref{AA-eq3} implies that
$$
\liminf_{t\to\infty}\inf_{-l+\delta\le x\le l-\delta,c+\varepsilon\le \bar c\leq c^*-\varepsilon}\tilde  u(t,x;u_0)>0,  \quad \forall \, 0<\delta <l.$$
Hence
$$
\liminf_{t\to\infty}\inf_{-l+\bar ct+\delta\le x\le l+\bar ct-\delta, c+\varepsilon\leq \bar c\leq c^*-\varepsilon} u(t,x;u_0)>0,  \quad \forall \,\, 0<\delta <l.$$
It then follows that
$$
\liminf_{t\to\infty}\inf_{-l+(c+\varepsilon)t+\delta\le x\le l+(c^*-\varepsilon)t-\delta} u(t,x;u_0)>0,  \quad \forall \, 0<\delta <l.$$
and
$$
\liminf_{t\to\infty}\inf_{(c+\varepsilon)t\le x\le (c^*-\varepsilon)t} u(t,x;u_0)>0.
$$

Finally,  suppose that $2\chi\mu<b$. We prove
\begin{equation}\label{tm1.1-2}
\lim_{t\to\infty} \sup_{(c+\varepsilon)t\le x\le (c^*-\varepsilon)t}|u(t,x;u_0)-\frac{r^*}{b}|=0.
\end{equation}
Suppose  by contraction that the result does not hold. Then there are constants $0< \varepsilon< \frac{c^*-c}{2}$, $\delta>0$, and a sequence $\{(x_n,t_n)\}_{n \in \N}$ such that $t_n \to \infty$, $t_n(c+\varepsilon)\leq x_n\leq t_n(c^*-\varepsilon)$, and
\begin{equation}\label{eq000}
|u(t_n,x_n; u_0)-\frac{r^*}{b}|\geq \delta \quad  \forall\, n\geq 1.
\end{equation}
For every $n\geq 1$, define $(u_n(t,x),v_n(t,x))=(u(t+t_n,x+x_n;u_0),v(t+t_n,x+x_n;u_0))$. By a priori  estimates for parabolic equations, without loss of generality,  we may suppose that $(u_n(t,x),v_n(t,x))\to (u^*(t,x),v^*(t,x))$ locally uniformly in $C^{1,2}(\R\times\R)$. Furthermore, $(u^*(t,x)$, $v^*(t,x))$  is an entire solution of \eqref{Le2.6-eq1}.

Choose $0<\tilde\varepsilon<\varepsilon<\frac{c^*-c}{2} $. For every $x\in \R$ and $t\in \R$, we have
\begin{align}
x+x_n&\leq x+t_n(c^*-\varepsilon) \cr
&=(c^*-\tilde\varepsilon)(t_n+t)-(\varepsilon-\tilde\varepsilon)(t_n-\frac{x-(c^*-\tilde\varepsilon)t}{\varepsilon-\tilde\varepsilon})\cr
&\leq (t+t_n)(c^*-\tilde\varepsilon)
\end{align}
whenever $t_n\geq \frac{\|x\|+(c^*-\tilde\varepsilon)|t|}{\varepsilon-\tilde\varepsilon}$.
On the other hand, For every $x\in \R$ and $t\in \R$, we have
\begin{align}
x+x_n&\geq x+t_n(c+\varepsilon) \cr
&=(c+\tilde\varepsilon)(t_n+t)+(\varepsilon-\tilde\varepsilon)(t_n-\frac{(c+\tilde\varepsilon)t-x}{\varepsilon-\tilde\varepsilon})\cr
&\geq (t+t_n)(c+\tilde\varepsilon)
\end{align}
whenever $t_n\geq \frac{\|x\|+(c+\tilde\varepsilon)|t|}{\varepsilon-\tilde\varepsilon}$.
Thus, it follows that
$$(t+t_n)(c+\tilde\varepsilon)\leq x+x_n\leq (t+t_n)(c^*-\tilde\varepsilon)$$
whenever $t_n\geq \max\{\frac{\|x\|+(c^*-\tilde\varepsilon)|t|}{\varepsilon-\tilde\varepsilon}, \frac{\|x\|+(c+\tilde\varepsilon)|t|}{\varepsilon-\tilde\varepsilon} \}$.
Note that
$$
u^*(t,x)=\lim_{n \to \infty}u(t+t_n,x+x_n;u_0)\geq \liminf_{s \to \infty}\inf_{s(c+\tilde\varepsilon)\leq y\leq s(c^*-\tilde\varepsilon)}u(s,y;u_0) >0
$$
for every $(t,x)\in \R \times \R$. Hence $\inf_{(t,x)\in \R \times \R } u^*(t,x)>0$.
By Lemma \ref{lem-001-4},
we must have $u^*(t,x)=\frac{r^*}{b}$ for every $(t,x)\in \R\times\R$. {In particular}, $u^*(0,0)=\frac{r^*}{b}$, which contradicts to \eqref{eq000}.

\medskip

 (3) First, let $\kappa=\sqrt{r^*}$,
 for any $\bar c$ satisfies
$$
c^*=2\sqrt {r^*} =c_\kappa<\bar c<|c|.
$$
By Lemma \ref{lem-001-1}(i),
$$
\lim_{t\to\infty}\sup_{|x|\ge \bar ct }  u(t,x;u_0)=0.
$$
This implies that for any $\varepsilon>0$,
$$
\lim_{t\to\infty}\sup_{|x|\ge (c^*+\varepsilon)t }  u(t,x;u_0)=0.
$$

Next,  for any $0<\varepsilon<c^*$,  {let $\bar r$ and  $l$ be as in \eqref{bar-r-eq} and
\eqref{l-eq}.
Fix a $\bar c$ satisfying ${ -c^*+\varepsilon\leq}  \bar c\le c^*-\varepsilon$. Let $\lambda(\bar c,\bar r)$ be as in
\eqref{lambda-eq}.}
By the similar arguments as those in the proof of \eqref{Tm1.1.2liminf}, it can be proved that
$$
\liminf_{t\to\infty}\inf_{-l+\bar ct+\delta\le x\le l+\bar ct-\delta, -c^*+\varepsilon\leq \bar c\leq c^*-\varepsilon} u(t,x;u_0)>0,  \quad \forall \, 0<\delta <l.
$$
Hence
$$
\liminf_{t\to\infty}\inf_{-l+(-c^*+\varepsilon)t+\delta\le x\le l+(c^*-\varepsilon)t-\delta} u(t,x;u_0)>0, \quad \forall \, 0<\delta <l$$
and then
$$
\liminf_{t\to\infty}\inf_{(-c^*+\varepsilon)t\le x\le (c^*-\varepsilon)t} u(t,x;u_0)>0.
$$
Moreover,  using the similar arguments as those in the proof of \eqref{tm1.1-2}, we can prove
if $2\chi\mu<b$, then
 $$
\lim_{t\to\infty} \sup_{(-c^*+\varepsilon)t\le x\le (c^*-\varepsilon)t}|u(t,x;u_0)-\frac{r^*}{b}|=0.
$$
\end{proof}

Now we prove   Theorem \ref{right-support}.

\begin{proof}[Proof of Theorem \ref{right-support}]
(1)   { Suppose $c\geq-c^*=-2\sqrt{r^*}$.} By the same arguments as those in Theorem \ref{compact-support}(2),
it can be proved that
$$\lim_{t\to\infty}\sup_{x\leq(c-\varepsilon)t}u(t,x;u_0)=0.$$

Next, we prove that
\begin{equation}\label{Tm1.2-1liminf-c}
\liminf_{t\to\infty}\inf_{x\ge  (c+\varepsilon)t}u(t,x;u_0)>0\quad \forall\,  \varepsilon>0.
\end{equation}

To this end, for any $\tilde \varepsilon>0 $, let $u(t,x;u_0)=\tilde u(t,x-(c+\tilde \varepsilon)t;u_0), v(t,x;u_0)=\tilde v(t,x-(c+\tilde \varepsilon)t;u_0) $ in \eqref{Keller-Segel-eq0}  and set $M=\max\{\|u_0\|_{\infty},\frac{r^*}{b-\chi\mu}\}$. By \eqref{E2-1}, it follows that, for any  $R\gg 1,{ p>1}$, $(\tilde u(t,x;u_0),\tilde v(t,x;u_0))$ satisfies
\begin{equation}\label{AA-eq1}
\tilde u_t\geq \tilde u_{xx}+(c+\tilde \varepsilon) \tilde u_x-\chi \tilde v_x\tilde u_x+\tilde u(r(x+\tilde \varepsilon t)-\varepsilon_R M- C_{R,p} \tilde u^{\frac{1}{p}}-(b-\chi\mu)\tilde u), \quad t\geq 1,\ x\in\R.
\end{equation}
Let $p=2$. Choose  $R\gg 1$, $0<\xi\ll 1$ { and $0<\epsilon\ll \min\{1, 2\sqrt{r^*}\}$}
such that $
\varepsilon_{R}M<\frac{\xi}{4}$,
\begin{equation}\label{AA-eq2-1}
|\chi \tilde v_x|\leq C_{R}\sqrt {\tilde u(t,x;u_0)}+\frac{\xi}{4}, \quad t\geq 1,\ x\in\R,
\end{equation}
and
\begin{equation}\label{AA-eq2-New1}
-c^*+\epsilon=-2\sqrt{r^*}+\epsilon \leq c +\tilde \varepsilon -\xi.
\end{equation}
Define
$$
B(t,x)=\frac{\chi \tilde v_x}{\max\{1, |\chi \tilde v_x|\xi^{-1}\}}, \quad t\geq 1,\ x\in\R.
$$
From this point, the remaining part of the proof is completed in four steps.

\smallskip

\noindent{\bf Step 1.} {\it  In this step we construct some sub-solution for \eqref{AA-eq1}.}

\smallskip

First, choose $0<\xi_1\ll 1$  satisfying
 \begin{equation}\label{AA-eq2-2}
 C_R\sqrt{\xi_1}+\xi_1<\frac{\xi}{4}.
 \end{equation}
Next, let $l$ be chosen as in \eqref{l-eq}.  Let $T>1$ be such that $r(x+\tilde \varepsilon t)\ge r^*-\varepsilon_R M$ for $x\ge -l$ and $t\ge T$.
 Let $\underline u_1(\cdot)\in C^b_{\rm unif}([-l,\infty))\setminus\{0\}$ be such that
$$
\underline u_1(-l)=0,\quad  \underline u^{'}_1(x)>0, \quad {\rm and} \quad \underline u_1(x)<\tilde u(T,x;u_0), \quad \forall \,\, x\geq -l
$$
 Choose $ -2\sqrt{r^*}+\epsilon \leq\bar c\leq\min\{2\sqrt{r^*}-\epsilon,(c+\tilde \varepsilon)-\xi\}$. Let $\underline u(t,x)$ be the solution of
\begin{equation}
\label{aux-new-eq011}
\begin{cases}
\underline u_t=\underline u_{xx}+\bar c\underline u_x+\underline u (r^* -2\varepsilon_R M- C_{R,p} \underline u^{\frac{1}{p}}-{ (\frac{r^*}{\xi_1}+(b-\chi\mu) ) } \underline u), \quad t >T,\ x>-l \cr
\underline u(t,-l)=0 \cr
\underline u(T,x)=\frac{\xi_1}{M+\xi_1}\underline u_1(x).
\end{cases}
\end{equation}

Note that $\underline u(t,x)\equiv \xi_1$ is a super-solution of \eqref{aux-new-eq011} and $\|\underline{u}(T,\cdot)\|_{\infty}<\xi_1$. Thus, by the comparison principle for parabolic equations that
$$
\underline{u}(t,x)<\xi_1, \quad\forall\ t\geq T, \ x\geq-l.
$$
Since $\underline u^{'}_1(x)>0$, we have $\underline u_{x}(t,x)>0$ { for any $t\geq T$, $x\geq -l$.}
 { Note that $|B(t,x)|<\xi$ for all $t\geq 1$, $x\in\R$. }
Thus $\underline u(t,x)$ satisfies
\begin{align*}
\underline{u}_t= &\underline u_{xx}+\bar c\underline u_x+\underline u (r^* -2\varepsilon_R M- C_{R,p} \underline u^{\frac{1}{p}}-{ (\frac{r^*}{\xi_1}+(b-\chi\mu) ) } \underline u)\cr
\le&\underline u_{xx}+(c+\tilde \varepsilon-B(t,x))\underline u_x+\underline u (r^* -2\varepsilon_R M- C_{R,p} \underline u^{\frac{1}{p}}-(b-\chi\mu)\underline u) \quad\forall\ t>T, \ x>-l.
\end{align*}

\smallskip

\noindent {\bf Step 2.} {\it In this step, we show that}
\begin{equation}\label{AA-Stp2eq1}
\liminf_{t\to\infty}\inf_{x\ge -l+\delta} \underline u(t,x)>0,{\quad  \forall \,\, \delta>0.  }
\end{equation}

\smallskip

{
Choose $\bar b>\frac{r^*}{\xi_1}+b$ and also }$\bar b\gg 1$ such that $u(T{ +1},x;\bar c,\bar b)<\underline u(T { +1} ,x)$  for $-l\le x\le l$, where $u(t,x;\bar c,\bar b)$ is the unique positive entire solution of \eqref {aux-new-eq01} with $A(t,x)=0$. { Fix such $\bar b$.}
 It follows from the comparison principle for parabolic equations that
 $$
 u(t,x;\bar c,\bar b)<\underline u(t,x), \quad t>T{ +1} ,\ -l<x<l.
 $$
 Repeating the same procedure, by induction, we get
 $$
 u(t,x-kl;\bar c,\bar b)<\underline u(t,x), \quad t>T{ +1}  ,\ (k-1)l<x<(k+1)l, \ k=0,1,2,\cdots
  $$
 There exists $\delta_0 >0$, such that
 $$\inf_{(k-1)l+\delta \leq x\leq (k+1)l-\delta, t\in\R}u(t,x-kl;\bar c,\bar b)>\delta_0,$$
 for any $0<\delta<l, \ k=0,1,2,\cdots$

 Therefore, we have
 $$
 \liminf_{t\to\infty}\inf_{x\geq -l+\delta}\underline u(t,x)>0,{\quad  \forall \,\, \delta>0.  }
 $$

 \noindent {\bf Step 3.} {\it  In this step we  show that}
\begin{equation}\label{AA-Ste3eq1}
\underline{u}(t,x)\leq \tilde u(t,x;u_0), \quad \forall\ x\geq -l , \ t\ge T.
\end{equation}

 First, note that
$$
\underline{u}(t,-l)=0<\tilde u(t,-l;u_0)\quad \forall\, \, t\ge T
$$
and
$$
\underline{u}(T,x)<\tilde u(T,x;u_0)\quad \forall\, x\ge -l.
$$
Note also that
$$
\tilde u_\infty(T):=\liminf_{x\to \infty}\tilde u(T,x;u_0)>\underline{u}_\infty(T):=\lim_{x\to\infty}\underline{u}(T,x).
$$
For given $t\ge T$, let
$$
\tilde u_\infty(t)=\liminf_{x\to \infty}\tilde u(t,x;u_0),\quad \underline{u}_\infty(t)=\lim_{x\to\infty}\underline{u}(t,x).
$$
We claim that
$$
\tilde u_\infty (t)>\underline{u}_\infty(t)\quad \forall\, t>T.
$$
In fact, for any given $t_0>T$, there is $x_n\to\infty$ such that
$$
\tilde u(t_0,x_n;u_0)\to \tilde u_\infty(t_0),\quad \underline{u}(t_0,x_n)\to \underline{u}_\infty(t_0)
$$
as $n\to\infty$.
Without loss of generality, we may assume that
$$
\tilde u(t,x+x_n;u_0)\to  \tilde u^*(t,x),\quad \underline{u}(t,x+x_n)\to \underline{u}^*(t)
$$
as $n\to\infty$ locally uniformly in $(t,x)\in (T,\infty)\times \R$.
By $\tilde u_\infty(T)>\underline{u}_\infty(T)$ and the comparison principle for parabolic equations, we have
$$
\tilde u^*(t,x)>\underline{u}^*(t)\quad \forall\, t>T,\quad x\in\R.
$$
In particular, we have
$$
\tilde u_\infty(t_0)=\tilde u^*(t_0,0)>\underline{u}_\infty(t_0).
$$
Hence the claim holds true.

Next, assume that there are $t>T$ and $x>-l$ such that $\tilde u(t,x;u_0)<\underline{u}(t,x)$.
Then there is $t_{\inf}>T$ such that
$$
\tilde u(t,x;u_0)>\underline{u}(t,x)\quad \forall\, T\le t<t_{\inf},\,\, x\ge -l
$$
and
$$
\inf_{x\ge -l}\big(\tilde u(t_{\inf},x;u_0)- \underline u(t_{\inf},x)\big)=0.
$$
By the above claim, there is $x_{\inf}\in (-l,\infty)$ such that
$$
\tilde u(t_{\inf},x_{\inf};u_0)=\underline{u}(t_{\inf},x_{\inf}).
$$
Then the similar arguments as those in the proof of \eqref{AA-eq3}, we have
$$
\tilde u(t_{\inf},x_{\inf};u_0)>\underline u(t_{\inf},x_{\inf}),
$$
which is a contradiction. Hence \eqref{AA-Ste3eq1} holds.

\smallskip

\noindent {\bf Step 4.} {\it  In this step, we prove \eqref{Tm1.2-1liminf-c}.}

By \eqref{AA-Stp2eq1} and \eqref{AA-Ste3eq1}, we deduce that $$
\liminf_{t\to\infty}\inf_{x\ge -l+\delta} \tilde u(t,x;u_0)>0, {\quad  \forall \,\, \delta>0.  }
$$
Since $u(t,x;u_0)=\tilde u(t,x-(c+\tilde \varepsilon)t)$, we have
$$
\liminf_{t\to\infty}\inf_{x\ge -l+(c+\tilde \varepsilon)t+\delta} u(t,x;u_0)>0\quad \forall \, \tilde \varepsilon>0, {\quad  \forall \,\, \delta>0.  }
$$
Hence,
$$
\liminf_{t\to\infty}\inf_{x\ge (c+\varepsilon)t} u(t,x;u_0)>0\quad\forall\, \varepsilon>0.
$$

\smallskip

Finally, we prove that  $$
\lim_{t\to\infty} \sup_{x\geq (c+\varepsilon)t}|u(t,x;u_0)-\frac{r^*}{b}|=0.
$$
 It can be proved by similar arguments as those in the proof of \eqref{tm1.1-2}.

 \medskip

 (2) Suppose $c<-c^*=-2\sqrt{r^*}$.
First, let $\kappa=\sqrt{r^*}$,
 for any $\bar c$ satisfies
$$
c<\bar c<-c^*=-2\sqrt {r^*}=-c_{\kappa}.
$$
By Lemma \ref{lem-001-1}(ii),
$$
\lim_{t\to\infty}\sup_{x\le \bar ct }  u(t,x;u_0)=0.
$$
This implies that for any $\varepsilon>0$,
$$
\lim_{t\to\infty}\sup_{x\le (-c^*-\varepsilon)t }  u(t,x;u_0)=0.
$$

Next, we prove
\begin{equation}\label{right-support-1}
\liminf_{t\to\infty}\inf_{x\ge  (-c^*+\varepsilon)t}u(t,x;u_0)>0\quad\forall\, \varepsilon>0.
\end{equation}
It suffices to prove that for any $0<\varepsilon<2\sqrt{r^*}$, \eqref{right-support-1} holds.

Let $0<\varepsilon< 2\sqrt{r^*}$ be given , let $u(t,x;u_0)=\tilde u(t,x-(-c^*+\varepsilon)t;u_0), v(t,x;u_0)=\tilde v(t,x-(-c^*+\varepsilon)t;u_0) $ in \eqref{Keller-Segel-eq0}  and set $M=\max\{\|u_0\|_{\infty},\frac{r^*}{b-\chi\mu}\}$. By \eqref{E2-1}, it follows that, for any  $R\gg 1,{ p>1}$, $(\tilde u(t,x;u_0),\tilde v(t,x;u_0))$ satisfies
\begin{equation}\label{3AA-eq1}
\tilde u_t\geq \tilde u_{xx}+(-c^*+\varepsilon) \tilde u_x-\chi \tilde v_x\tilde u_x+\tilde u(r(x+(-c^*+\varepsilon-c) t)-\varepsilon_R M- C_{R,p} \tilde u^{\frac{1}{p}}-(b-\chi\mu)\tilde u), \quad t\geq 1,\ x\in\R.
\end{equation}
Choose $0<\xi\ll \min\{1,\frac{\varepsilon}{2}\}$. Fix a $\bar c$ satisfying $-c^*+\frac{\varepsilon}{2}\leq \bar c\leq -c^*+\varepsilon-\xi$.
By the similar arguments as those in the proof of \eqref{Tm1.2-1liminf-c}, it can be proved that
$$
\liminf_{t\to\infty}\inf_{x\ge -l+\delta} \tilde u(t,x;u_0)>0, {\quad  \forall \,\, \delta>0.  }
$$

Since $u(t,x;u_0)=\tilde u(t,x-(-c^*+\varepsilon)t)$, we have
$$
\liminf_{t\to\infty}\inf_{x\ge -l+(-c^*+\varepsilon)t+\delta} u(t,x;u_0)>0, {\quad  \forall \,\, \delta>0.  }$$
Hence,
$$
\liminf_{t\to\infty}\inf_{x\ge (-c^*+\varepsilon)t} u(t,x;u_0)>0.
$$

\smallskip

Finally, we prove $$
\lim_{t\to\infty} \sup_{x\geq (-c^*+\varepsilon)t}|u(t,x;u_0)-\frac{r^*}{b}|=0.
$$
It can be proved using similar arguments as those in the proof of \eqref{tm1.1-2}.
\end{proof}

\section{Persistence and extinction in Case 2}

In this section, we study the persistence and extinction of solutions of \eqref{Keller-Segel-eq0} with $r(x)$ being as in {\bf Case 2}, and prove
Theorem \ref{persistence-extinction-thm1}.  Throughout this section, we assume that {\bf (H1)} holds and $r(x)$ is as in {\bf Case 2}.

We first prove some lemmas.

Consider
\begin{equation}
\label{aux-new-new-new-eq01}
\begin{cases}
u_t=u_{xx}+c u_x-A(t,x)u_x+u (r(x) - \varepsilon_R M- C_{R,p} u^{\frac{1}{p}}-(\bar b-\chi\mu)u), \quad -L<x<L\cr
u(t,-L)=u(t,L)=0.
\end{cases}
\end{equation}
Observe that $\varepsilon_R \to 0$ as $R\to\infty$ and
$$
\lambda_L(r(\cdot)-2\varepsilon_R M)=\lambda_L(r(\cdot))-2\varepsilon_R M,
$$
where $\lambda_L(\cdot)$ is defined as in \eqref{ev-eq0}.
Hence
$$
\lim_{R\to\infty} \lambda_L(r(\cdot)-2\varepsilon_R M)=\lambda_L(r(\cdot)).
$$

\begin{lem}
\label{new-new-lm1}
Suppose that $\lambda_\infty(r(\cdot))>0$. Then  there are $L^*>0$, $R^*>0$, and $\epsilon^*>0$  such that for any $L\ge L^*$ and
$R\ge R^*$, $\|A(\cdot,\cdot)\|_\infty<\epsilon^*$, \eqref{aux-new-new-new-eq01} has a {unique} positive {bounded} entire solution
$u^*(t,x;L,R,A(\cdot,\cdot))$ with
\begin{equation}
\label{aux-new-new-new-eq2}
\inf_{t\in\R,|x|\le L-\delta} u^*(t,x;L,R,A(\cdot,\cdot))>0\quad \forall\,\, 0<\delta<L.
\end{equation}
\end{lem}

\begin{proof}
First of all, there are $L^*>0$ and $R^*>0$ such that
$$
\lambda_L(r(\cdot)-2\varepsilon_R M)>0\quad \forall\,\, L\ge L^*,\,\, R\ge R^*.
$$

It then follows from  similar arguments as those in Lemma \ref{lem-001-3} that there is $\epsilon^*>0$ such that for any
$A(\cdot,\cdot)$ with $\|A\|_\infty\le \epsilon^*$,  \eqref{aux-new-new-new-eq01} has a {unique}
positive {bounded}
entire solution
$u^*(t,x;L,R,A(\cdot,\cdot))$ satisfying  \eqref{aux-new-new-new-eq2}.
\end{proof}

For every $u\in C^b_{\rm unif}(\R)$, let
\begin{equation}\label{psi-definition}
\Psi(x;u)=\mu\int_{0}^{\infty}\int_{\R}\frac{e^{-\nu s}e^{-\frac{|y-x|^2}{4s}}}{\sqrt{4\pi s}}u(y)dyds.
\end{equation}
It is well known that $\Psi(x;u)\in C^2_{\rm unif}(\R)$ and solves the elliptic equation
\begin{equation}\label{Psi-eq}
\frac{d^2}{dx^2}\Psi(x;u)-\nu\Psi(x;u)+\mu u=0.
\end{equation}

\begin{lem}\label{new-new-lm2}
For every $u\in C^b_{\rm unif}(\R)$, $u(x)\ge 0$, it holds that
\begin{equation}\label{estimate-on-space-derivative-1}
\left| \frac{d}{dx}\Psi(x;u)\right|\leq \sqrt{\nu}\Psi(x;u),\ \quad  \forall\ x\in\R.
\end{equation}
\end{lem}

\begin{proof}
The lemma is proved in \cite[Lemma 2.2]{SaShXu}.
\end{proof}

Let $\phi_L(x)$ be the positive principal eigenfunction of \eqref{ev-eq0} corresponding to the principal eigenvalue $\lambda_L(r(\cdot))$
with $\phi_L(0)=1$.  By a priori estimates and Harnack's inequality for elliptic equations, there exist $L_n\to\infty$ and  $\phi_{\infty}(x)>0$
 such that
 $$
 \lim_{n\to\infty} \phi_{L_n}(x)=\phi_\infty(x)
 $$
 locally uniformly, and
\begin{equation}\label{phi_1}
(\phi_{\infty})_{xx}+c(\phi_{\infty})_{x}+ r(x) \phi_{\infty}=\lambda_{\infty} \phi_{\infty},\quad x\in\R.
\end{equation}

{
\begin{lem}
\label{new-new-lm3}
\begin{equation}\label{phi_derivi_est}
|\frac{d}{dx}\phi_{\infty}(x)|\leq \frac{\sqrt{8r^*+c^2}+|c|}{2}\phi_{\infty}(x)  \quad \forall\, x\in\R.
\end{equation}
\end{lem}

\begin{proof}
It follows from  \cite[Lemma 2.1, Lemma 2.2]{SaSh4}. To be more precise, first
 rewrite $\eqref{phi_1}$ as
\begin{equation}\label{phi_2}
(\phi_{\infty})_{xx}+c(\phi_{\infty})_{x}-2r^*\phi_{\infty}+\big(2r^*+r(x)-\lambda_{\infty}\big)\phi_{\infty}= 0,\quad x\in\R.
\end{equation}
Put $a(x):=\big(2r^*+r(x)-\lambda_{\infty}\big)\phi_{\infty}(x)$. Note that $a(x)\ge 0$ for any $x\in\R$.

Next, it can be verified directly that
$$
\phi_\infty (x)=\int_0^\infty  \int_{\R}\frac{e^{-2r^* s}}{\sqrt{4\pi s}}e^{-\frac{|x+cs-y|^2}{4s}}a(y)dy ds.
$$
By \cite[Lemma 2.1]{SaSh4},
$$
\phi_\infty(x)=\frac{1}{\sqrt{8r^*+c^2}}\Big( e^{-\lambda_1^cx} \int_{-\infty}^x e^{\lambda_1^cy}a(y)dy+e^{\lambda_2^c x}\int_x^\infty e^{-\lambda_2^c y}a(y)dy\Big)
$$
and
\begin{equation}
\label{revised-eq1}
\frac{d}{dx}\phi_\infty(x)=\frac{1}{\sqrt{8r^*+c^2}}\Big(-\lambda_1^c  e^{-\lambda_1^c x} \int_{-\infty}^x e^{\lambda_1^cy}a(y)dy+\lambda_2^c e^{\lambda_2^c x}\int_x^\infty e^{-\lambda_2^c y}a(y)dy\Big),
\end{equation}
where
$$
\lambda_1^c =\frac{\sqrt{8r^*+c^2}+c}{2},\quad \lambda_2^c=\frac{\sqrt{ 8r^*+c^2}-c}{2}.
$$

Now, by $a(x)\ge 0$ and \eqref{revised-eq1}, we have
\begin{equation*}
|\frac{d}{dx}\phi_{\infty}(x)|\leq \frac{\sqrt{8r^*+c^2}+|c|}{2}\phi_{\infty}(x)  \quad \forall\, x\in\R.
\end{equation*}
The lemma is thus proved.
\end{proof}}

We now prove Theorem \ref{persistence-extinction-thm1}.

\begin{proof} [Proof of Theorem \ref{persistence-extinction-thm1}]

(1)
 If $c>c^*$, using the same arguments as those in the proof of Theorem \ref{compact-support} (1), we can prove
$\lim_{t\to\infty}u(t,x;u_0)=0$ uniformly for $x\in\R$.

If $c<-c^*$, let $\kappa=\sqrt{r^*}$,
 for any $\bar c$ satisfies
$$
c<\bar c<-c^*=-2\sqrt {r^*}=-c_{\kappa}.
$$
By Lemma \ref{lem-001-1}(ii),
$$
\lim_{t\to\infty}\sup_{x\le \bar ct }  u(t,x;u_0)=0.
$$
We claim that $\lim_{t\to\infty}\sup_{x\in\R}  u(t,x;u_0)=0$. For otherwise, there are $\delta_0>0$, $t_n\to\infty$ and $x_n\in (\bar c t_n, \infty)$
such that
$$
 u(t_n,x_n;u_0)\ge \delta_0\quad \forall\, n\ge 1.
$$
Let $u_n(t,x)=u(t+t_n,x+x_n;u_0)$ and $v_n(t,x)= v(t+t_n,x+x_n;u_0)$. Note  that
  $x_n-c t_n\to  \infty$ as $n\to\infty$.
  Without loss of generality, we may assume that there is $(u^*(t,x),v^*(t,x))$ such that
$$
\lim_{n\to\infty}(u_n(t,x),v_n(t,x))=(u^*(t,x),v^*(t,x))
$$
locally uniformly in $(t,x)\in\R\times \R$, and $(u^*(t,x),v^*(t,x))$ satisfies
\begin{equation}\label{new-Keller-Segel-eq2}
\begin{cases}
u_t^*=\Delta u^*-\nabla \cdot (\chi u^* \nabla v^*)+u^*(r(\infty)-bu^*),\quad t\in\R,\,\, x\in\R\cr
0 =\Delta v^*-  \nu v^* +\mu u^*,\quad t\in\R,\,\, x\in\R.
\end{cases}
\end{equation}
By $r(\infty)<0$, it can be proved that $u^*(t,x)\equiv 0$, which contradicts to
$$u^*(0,0)=\lim_{n\to\infty}  u(t_n,x_n;u_0)\ge \delta_0.$$

Therefore, $\lim_{t\to\infty} \sup_{x\in\R}  u(t,x;u_0)=0$.
\end{proof}

\begin{proof} [Proof of Theorem \ref{persistence-extinction-thm1}]
(2) First, fix $u_0$ with {nonempty} compact support.
 Let $M$ be such that $M\geq \max\{\frac{r^*}{b-\chi\mu}, \sup_{x\in\R}u_0(x)\}$.
 Note  that
 $$
 u(t,x;u_0)\le M\quad \forall\,\, t\ge 0,\,\, x\in\R.
 $$

Next, let $\phi_\infty (x)$ be as in Lemma \ref{new-new-lm3}.
Without loss of generality, we may suppose that $u_0(x)\leq \phi_{\infty}(x)$ { for any $x\in\R$}.

Let $u_{\infty}(t,x)=e^{\lambda_{\infty}t}\phi_{\infty}(x)$. Then $u_{\infty}(t,x)$ satisfies the following parabolic equation
\begin{align}\label{u_infty-equ}
(u_{\infty})_{t}&=(u_{\infty})_{xx}+c(u_{\infty})_{x}+ r(x) u_{\infty}\cr
&\geq (u_{\infty})_{xx}+c(u_{\infty})_{x}+ (r(x)-(b-\chi\mu)u_{\infty} )u_{\infty},\, \,\forall\,  t>0,  x\in\R.
\end{align}
Hence, if $\nu \geq \nu^*:={\frac{(\sqrt{8r^*+c^2}+|c|)^2}{4}}$, then
by Lemma \ref{new-new-lm2} and \eqref{phi_derivi_est}, we get
\begin{align}
-\chi (u_{\infty})_{x} \Psi_x(x;u(t,\cdot;u_0))-\chi\nu\Psi(x;u(t,\cdot;u_0)) u_{\infty}&\leq \chi |(u_{\infty})_{x}| |\Psi_x|-\chi\nu\Psi u_{\infty}\cr
&\leq \chi\sqrt{\nu}\Psi (|(u_{\infty})_{x}|-\sqrt{\nu}u_{\infty})\cr
&=\chi\sqrt{\nu}\Psi e^{\lambda_{\infty}t}\big(|(\phi_{\infty})_{x}|-\sqrt{\nu}\phi_{\infty}\big)\cr
&\leq \chi\sqrt{\nu}\Psi e^{\lambda_{\infty}t}\big(\sqrt{\nu^*}- \sqrt{\nu}\big)\phi_{\infty}\cr
&\leq 0, { \,\,\, \forall\, x\in\R.}
\end{align}
By \eqref{u_infty-equ}, we have
\begin{align*}
(u_{\infty})_{t}\geq &(u_{\infty})_{xx}+c(u_{\infty})_{x}-\chi (u_{\infty})_{x} \Psi_x(x+ct ;u(t,\cdot;u_0))\\
&+ (r(x)-\chi\nu\Psi(x+ct;u(t,\cdot;u_0)) -(b-\chi\mu)u_{\infty} )u_{\infty},\, \forall x\in\R.
\end{align*}
{Let $\tilde u(t,x;u_0)=u(t,x+ct;u_0)$, then $\tilde u(t,x;u_0)$ satisfies
$$
\tilde u_{t}= \tilde u_{xx}+c \tilde u_{x}-\chi \tilde u_{x} \Psi_x(x+ct ;u(t,\cdot;u_0))+ (r(x)-\chi\nu\Psi(x+ct;u(t,\cdot;u_0)) -(b-\chi\mu)\tilde u )\tilde u,\, \forall \,t>0, x\in\R.
$$}
By the comparison principle for parabolic equations, we have
$$
\tilde u(t,x;u_0) \leq u_{\infty}(t,x)=e^{\lambda_{\infty}t}\phi_{\infty}(x), { \quad \forall\,\, t\ge 0,\,\, x\in\R.}
$$

Since $\phi_{\infty}(x)$ is bounded { on any compact set} and $\lambda_{\infty}<0$, we then have $\lim_{t\to\infty} \tilde u(t,x;u_0)=0$ locally  uniformly in $x\in\R$.

We prove now that $\lim_{t\to\infty} \tilde u(t,x;u_0)=0$   uniformly in $x\in\R$. Assume by contradiction that this is not true. Then
there is $\epsilon_0>0$, $t_n\to\infty$, and $|x_n|\to\infty$ such that
$$
\tilde u(t_n,x_n;u_0)\ge \epsilon_0.
$$
Without loss of generality, we assume that $x_n\to\infty$, and  $\lim_{n\to\infty}\tilde u(t+t_n,x+x_n;u_0)=U^*(t,x)$,
$ \lim_{n\to\infty}  \Psi(x+x_n+c{(t+t_n)}; { u}(t+t_n,\cdot;u_0))=\Psi^*(t,x)$ locally uniformly. Then
$$
U^*_t=U^*_{xx}+{ cU^*_{x}}-\chi \Psi^*_x U^*_x+U^*(r(\infty)-\chi\nu \Psi^*-(b-\chi\mu)U^*),\quad t\in\R,\,\, x\in\R.
$$
Note that $U^*(t,x)$ is bounded and nonnegative {and $r(\infty)<0$}. We must have
$$
U^*(t,x)\equiv 0,
$$
which contradicts to $U^*(0,0)\ge \epsilon_0$. Therefore,
$\lim_{t\to\infty} \tilde u(t,x;u_0)=0$   uniformly in $x\in\R$,
{ which implies that $\lim_{t\to\infty} u(t,x;u_0)=0$
uniformly in $x\in\R$.}
\end{proof}

\begin{proof} [Proof of Theorem \ref{persistence-extinction-thm1}]
(3)
Suppose $|c|<c^*$. We first prove
$$
\lim_{t\to\infty} \sup_{|x-ct|\ge c^{'}t}u(t,x;u_0)=0\quad \forall\,\, c^{'}>0.
$$
Assume that the result does not hold. Then there are {constants $\delta_0>0$, and a sequence
$\{(t_n, x_n)\}_{n\in\N}$,  $t_n\to\infty$, $|x_n-ct_n|\geq c^{'} t_n$
such that }
$$
u(t_n,x_n;u_0)\ge \delta_0\quad \forall\, n\ge 1.
$$
Let $u_n(t,x)=u(t+t_n,x+x_n;u_0)$ and $v_n(t,x)=v(t+t_n,x+x_n;u_0)$.    Note  that
  $|x_n-c t_n|\to  \infty$ as $n\to\infty$. Thus, either  $x_n-c t_n\to  \infty$ as $n\to\infty$ or $x_n-c t_n\to  -\infty$ as $n\to\infty$.

In the case $x_n-c t_n\to  \infty$ as $n\to\infty$.
{Following similar arguments as those in the proof of Theorem \ref{persistence-extinction-thm1} (1), we can get a contradiction.}

In the case $x_n-c t_n\to  -\infty$ as $n\to\infty$. Also using similar arguments as those in the proof of Theorem \ref{persistence-extinction-thm1} (1) and the fact $r(-\infty)<0$, we also can get a contradiction.

\smallskip

Next, we prove that,  if $\lambda_\infty(r(\cdot))>0$, then
$$
\liminf_{t\to\infty}\inf_{|x-ct|\le L} u(t,x;u_0)>0\quad \forall\,\, L>0.
$$
To this end, let $u(t,x;u_0)=\tilde u(t,x-ct;u_0), v(t,x;u_0)=\tilde v(t,x-ct;u_0) $ in \eqref{Keller-Segel-eq0}  and
set $M=\max\{\|u_0\|_{\infty},\frac{r^*}{b-\chi\mu}\}$.
By \eqref{E2-1}, it follows that, for any  $R\gg 1, { p>1} $, $(\tilde u(t,x;u_0),\tilde v(t,x;u_0))$ satisfies
\begin{equation}\label{Tm3AA-eq1}
\tilde u_t\geq \tilde u_{xx}+ c \tilde u_x-\chi \tilde v_x\tilde u_x+\tilde u(r(x)-\varepsilon_R M- C_{R,p} \tilde u^{\frac{1}{p}}-(b-\chi\mu)\tilde u), \quad t\geq 1,\ x\in\R.
\end{equation}
Let $p=2$ and $\epsilon^*$ be as in
Lemma \ref{new-new-lm1}.
{By the similar arguments as those in the proof of \eqref{Tm1.1.2liminf}, it can be proved that}
\begin{equation}\label{AA-eq3-1}
u^*(t,x; L,R, A(\cdot,\cdot))\leq \tilde u(t,x;u_0)\quad \forall\, t\ge 1,\,\, -L\le x\le L.
\end{equation}
By Lemma \ref{new-new-lm1}, there are $L^*>0$, and $R^*>0$ such that for any $L\ge L^*$,
$R\ge R^*$,  and $\|A(\cdot,\cdot)\|_\infty<\epsilon^*$,
$$
\inf_{t\in\R,|x|\le L-\delta} u^*(t,x;L,R,A(\cdot,\cdot))>0\quad \forall\,\, 0<\delta<L.
$$
This together with \eqref{AA-eq3-1} implies that
$$
\liminf_{t\to\infty}\inf_{|x|\le L-\delta}\tilde  u(t,x;u_0)>0,  \quad \forall \,\, 0<\delta <L.
$$
Since  $u(t,x;u_0)=\tilde u(t,x-ct;u_0)$, we then have
$$
\liminf_{t\to\infty}\inf_{|x-ct|\le L-\delta}  u(t,x;u_0)>0,  \quad \forall \,\, 0<\delta <L,
$$
which implies that
$$
\liminf_{t\to\infty}\inf_{|x-ct|\le L} u(t,x;u_0)>0\quad \forall\,\, L>0.
$$
\end{proof}


\end{document}